\documentclass[a4paper,12pt,reqno]{amsart}

\usepackage{fullpage}
\usepackage{amsmath,amssymb,graphicx,psfrag}
\usepackage[T1]{fontenc}
\usepackage{amsmath,amssymb,ifthen}
\usepackage{color}
\usepackage{pgfplots}

\usepgfplotslibrary{external}
\newcommand{\eps}{\varepsilon}

\newcommand{\logLogSlopeTriangle}[6]
{

    \pgfplotsextra
    {
        \pgfkeysgetvalue{/pgfplots/xmin}{\xmin}
        \pgfkeysgetvalue{/pgfplots/xmax}{\xmax}
        \pgfkeysgetvalue{/pgfplots/ymin}{\ymin}
        \pgfkeysgetvalue{/pgfplots/ymax}{\ymax}

        \pgfmathsetmacro{\xArel}{#1}
        \pgfmathsetmacro{\yArel}{#3}
        \pgfmathsetmacro{\xBrel}{#1-#2}
        \pgfmathsetmacro{\yBrel}{\yArel}
        \pgfmathsetmacro{\xCrel}{\xArel}

        \pgfmathsetmacro{\lnxB}{\xmin*(1-(#1-#2))+\xmax*(#1-#2)} 
        \pgfmathsetmacro{\lnxA}{\xmin*(1-#1)+\xmax*#1} 
        \pgfmathsetmacro{\lnyA}{\ymin*(1-#3)+\ymax*#3} 
        \pgfmathsetmacro{\lnyC}{\lnyA+#4*(\lnxA-\lnxB)}
        \pgfmathsetmacro{\yCrel}{\lnyC-\ymin)/(\ymax-\ymin)} 

        \coordinate (A) at (rel axis cs:\xArel,\yArel);
        \coordinate (B) at (rel axis cs:\xBrel,\yCrel);
        \coordinate (C) at (rel axis cs:\xCrel,\yCrel);

        \draw[#5]   (A)--
                    (B)-- 
                    (C)-- node[pos=0.5,anchor=west] {#6}
                    cycle;
    }
}

\newcommand{\logLogSlopeTriangleBelow}[6]
{

    \pgfplotsextra
    {
        \pgfkeysgetvalue{/pgfplots/xmin}{\xmin}
        \pgfkeysgetvalue{/pgfplots/xmax}{\xmax}
        \pgfkeysgetvalue{/pgfplots/ymin}{\ymin}
        \pgfkeysgetvalue{/pgfplots/ymax}{\ymax}

        \pgfmathsetmacro{\xArel}{#1}
        \pgfmathsetmacro{\yArel}{#3}
        \pgfmathsetmacro{\xBrel}{#1-#2}
        \pgfmathsetmacro{\yBrel}{\yArel}
        \pgfmathsetmacro{\xCrel}{\xArel}

        \pgfmathsetmacro{\lnxB}{\xmin*(1-(#1-#2))+\xmax*(#1-#2)} 
        \pgfmathsetmacro{\lnxA}{\xmin*(1-#1)+\xmax*#1} 
        \pgfmathsetmacro{\lnyA}{\ymin*(1-#3)+\ymax*#3} 
        \pgfmathsetmacro{\lnyC}{\lnyA+#4*(\lnxA-\lnxB)}
        \pgfmathsetmacro{\yCrel}{\lnyC-\ymin)/(\ymax-\ymin)} 

        \coordinate (A) at (rel axis cs:\xArel,\yArel);
        \coordinate (B) at (rel axis cs:\xBrel,\yCrel);
        \coordinate (C) at (rel axis cs:\xBrel,\yArel);

        \draw[#5]   (A)--
                    (B)-- node[pos=0.5,anchor=east] {#6}
                    (C)-- 
                    cycle;
    }
}

\newcommand{\dual}[2]{\langle#1\hspace*{.5mm},#2\rangle}
\newcommand{\ip}[2]{(#1\hspace*{.5mm},#2)}

\newcommand{\norm}[3][]{#1\|#2#1\|_{#3}}

\def\enorm#1{|\hspace*{-.5mm}|\hspace*{-.5mm}|#1|\hspace*{-.5mm}|\hspace*{-.5mm}|}

\newcommand{\diam}{\mathrm{diam}}

\def\div{{\rm div\,}}

\newcommand{\material}{A}
\newcommand{\xx}{\boldsymbol{x}}
\newcommand{\yy}{\boldsymbol{y}}

\newcommand{\Riesz}{\ensuremath{\mathcal{R}}}
\newcommand{\FE}{\ensuremath{\mathcal{A}}}

\newcommand{\dlo}{\ensuremath{\mathcal{K}}}
\newcommand{\slo}{\ensuremath{\mathcal{V}}}
\newcommand{\hyp}{\ensuremath{\mathcal{W}}}

\newcommand{\stpo}{\ensuremath{\mathcal{S}\!\mathcal{P}}}

\newcommand{\set}[2]{\left\lbrace{#1} \,:\, {#2} \right\rbrace}

\newcommand{\R}{\ensuremath{\mathbb{R}}}
\newcommand{\N}{\ensuremath{\mathbb{N}}}
\newcommand{\HH}{\ensuremath{\boldsymbol{H}}}

\newcommand{\LL}{\ensuremath{\mathcal{L}}}
\newcommand{\XX}{\ensuremath{\mathcal{X}}}
\newcommand{\YY}{\ensuremath{\mathcal{Y}}}
\newcommand{\II}{\ensuremath{\mathcal{I}}}

\newcommand{\normal}{{\boldsymbol{n}}}
\newcommand{\vv}{\ensuremath{\boldsymbol{v}}}

\newcommand{\TT}{\ensuremath{\mathcal{T}}}
\newcommand{\MM}{\ensuremath{\mathcal{M}}}
\renewcommand{\SS}{\ensuremath{\mathcal{S}}}

\newcommand{\PP}{\ensuremath{\mathcal{P}}}
\newcommand{\OO}{\ensuremath{\mathcal{O}}}

\newcommand{\uu}{\boldsymbol{u}}

\newcommand{\err}{\operatorname{err}}

\renewcommand{\H}{\widetilde{H}}
\newcommand{\RR}{\mathcal{R}}
\renewcommand{\SS}{\mathcal{S}}

\newcommand{\edual}[3][]{#1\langle\!\!#1\langle#2\,,\,#3#1\rangle\!\!#1\rangle}
\newcommand{\refine}{{\rm refine}}

\def\AA{\mathcal{A}}
\def\reff#1#2{\!\stackrel{\eqref{#1}}{#2}\!}

\newtheorem{lemma}{Lemma}
\newtheorem{theorem}[lemma]{Theorem}
\newtheorem{proposition}[lemma]{Proposition}

\newtheorem{remark}[lemma]{Remark}
\newtheorem{algorithm}[lemma]{Algorithm}

\renewcommand{\subsection}[1]{\refstepcounter{subsection}\medskip{\bf\thesubsection.~#1.}}

\usepackage{fancyhdr}
\advance\footskip0.4cm
\advance\textheight-0.4cm
\calclayout
\title{A linear Uzawa-type solver\\ for nonlinear transmission problems}

\author{Thomas F\"uhrer}
\author{Dirk Praetorius}

\address{TU Wien, Institute for Analysis and Scientific Computing, Wiedner Hauptstr.~8--10/E101/4, 1040 Wien, Austria}
\email{thomas.fuehrer@tuwien.ac.at\quad\rm(corresponding author)}
\email{dirk.praetorius@tuwien.ac.at}

\thanks{\textbf{Acknowledgement.} The authors acknowledge support through the Austrian Science Fund (FWF) under grant
P27005 \textit{Optimal adaptivity for BEM and FEM-BEM coupling}.}

\date{\today}

\begin{document}

\begin{abstract}
We propose an Uzawa-type iteration for the Johnson--N\'ed\'elec formulation of a Laplace-type transmission problem with possible (strongly monotone) nonlinearity in the interior domain. In each step, we sequentially solve one BEM for the weakly-singular integral equation associated with the Laplace-operator and one FEM for the linear Yukawa equation. In particular, the nonlinearity is only evaluated to build the right-hand side of the Yukawa equation. We prove that the proposed method leads to linear convergence with respect to the number of Uzawa iterations. Moreover, while the current analysis of a direct FEM-BEM discretization of the Johnson--N\'ed\'elec formulation requires some restrictions on the ellipticity (resp.\ strong monotonicity constant) in the interior domain, our Uzawa-type solver avoids such assumptions.
\end{abstract}
\maketitle


\section{Introduction}

\subsection{State of the art}
The mathematical understanding of convergence of adaptive algorithms even with optimal rates has matured. We refer to the seminal works~\cite{doerfler,mns,stevenson07,ckns,ffp14} for the adaptive finite element method (FEM),~\cite{fkmp13,gantumur} for the adaptive boundary element method (BEM), as well as to~\cite{axioms} for some abstract axiomatic framework. Convergence of the adaptive FEM-BEM coupling has been proved in~\cite{afp} for heuristic $h-h/2$ type error estimators as well as in~\cite{affkmp} for residual error estimators, while the proof of optimal convergence rates is still missing due to the lack of some crucial orthogonality property (which is so far only known for elliptic problems which are symmetric up to a compact perturbation; see~\cite{ffp14,bhp17}). For linear problems, the influence of the inexact (iterative) solution of the Galerkin systems on the (optimal) convergence of adaptive FEM is analyzed in~\cite{agl13,ffp14}, while adaptive inexact FEM for strongly monotone problems has been considered in~\cite{ghps16}.

In the present work, we consider a possibly nonlinear transmission problem on the full space $\R^d$, $d\ge2$. Then, the FEM-BEM coupling is often the method of choice, since it allows to handle the nonlinearity in the bounded FEM domain, while properly treating the radiation condition at infinity. We follow an idea from~\cite{stokesUzawa} for the linear Stokes problem and transfer their approach to the nonlinear FEM-BEM coupling. Unlike~\cite{ghps16}, we do not use a Picard iteration to linearize the coupled system, but employ an Uzawa-type outer iteration combined with adaptive mesh-refinement for both the FEM part and the BEM part of the coupled system in an inner iteration. One particular benefit of the proposed approach is that each step of the Uzawa iteration only considers either a linear and symmetric FEM problem or a linear BEM problem, despite the nonlinearity (or the non-symmetry) of the transmission problem resp.\ of its FEM-BEM coupling formulation. Moreover, our analysis of the proposed algorithm also allows the inexact (iterative) solution of these FEM or BEM problems. In particular, we employ only standard preconditioning techniques for the FEM \emph{or} the BEM on adaptively refined meshes, while no preconditioner for the coupled FEM-BEM system is required.

Throughout, our focus is on the Johnson-N\'ed\'elec FEM-BEM coupling formulation~\cite{johned}. Unlike the so-called symmetric coupling~\cite{han90,costabel} which involves all four boundary integral integral operators associated with the partial differential equation in the exterior domain, the Johnson-N\'ed\'elec coupling relies only on the simple-layer and the double-layer integral operator and, in particular, avoids the so-called hypersingular integral operator. For this reason, the Johnson-N\'ed\'elec coupling is often preferred in practice, even if the symmetric coupling is better understand from the point of numerical analysis.

We stress that well-posedness of the Johnson-N\'ed\'elec coupling on polygonal domains has only been proved recently in the seminal work~\cite{sayas09} for linear Laplace- and Yukawa-type transmission problems (see also~\cite{s2,os,os2014} for general linear problems), while stability for nonlinear problems has been treated in~\cite{affkmp,FBlame}. Throughout, the current analysis requires that the ellipticity (resp.\ monotonicity) in the FEM domain is sufficiently large (see~\cite{s2,os,os2014,affkmp,FBlame}), since the proof of the discrete {\rm inf--sup} condition (resp.\ discrete monotonicity estimate) essentially relies on energy arguments. Numerical experiments in~\cite{affkmp}, however, indicate that this might not be necessary in practice.

In any case, it is worth noting that the present Uzawa-type algorithm only requires well-posedness of the continuous problem. Our analysis avoids any additional assumption on the validity of the discrete {\rm inf--sup} condition  (resp.\ discrete monotonicity). In explicit terms, the proposed algorithm is proved to be stable, even if the pair of discrete FEM and BEM spaces would not yield a positive {\rm inf--sup} constant and would thus be unstable for the direct solution of the Johnson-N\'ed\'elec FEM-BEM coupling. Numerical experiments give evidence for optimal convergence behavior of the proposed algorithm even if the (unknown) exact solution of the transmission problem has singularities.

\subsection{Model problem}
With $d\ge2$, let $\Omega\subseteq\R^d$ be a bounded and simply connected Lipschitz domain with boundary $\Gamma := \partial\Omega$ and normal vector $\normal\in L^\infty(\Gamma)$ pointing from $\Omega$ to the unbounded domain $\Omega^{\mathrm{ext}} := \R^d \setminus \overline\Omega$. Given $f\in L^2(\Omega)$, $u_0\in H^{1/2}(\Gamma)$, and $\phi_0 \in H^{-1/2}(\Gamma)$, we seek the solution $(u,u^\mathrm{ext})$ of the transmission problem
\begin{subequations}\label{eq:modell}
\begin{alignat}{2}
 \AA u := -\div(A(\nabla u)) + b(\nabla u) + c(u)
  &= f &\quad&\text{in }\Omega,\label{eq:modell:int} \\
  -\Delta u^{\mathrm{ext}} &= 0 &\quad&\text{in } \Omega^{\mathrm{ext}},
  \label{eq:modell:ext} \\
  u-u^{\mathrm{ext}} &= u_0 &\quad&\text{on } \Gamma, \label{eq:modell:jump} \\
  (\material\nabla u - \nabla u^{\mathrm{ext}})\cdot\normal & = \phi_0 &\quad& \text{on } \Gamma,
  \label{eq:modell:jumpn}
\end{alignat}
where the behavior at infinity is prescribed by
\begin{alignat}{2}
  u^{\mathrm{ext}}(\xx) &= 
  \begin{cases}
      C_\mathrm{rad} \log|\xx| + \OO(|\xx|^{-1}) &\text{for }d=2, \\
      \OO(|\xx|^{-1}) &\text{for }d=3, 
    \end{cases}
    &\quad&\text{as } |\xx|\to\infty \label{eq:modell:rad}
\end{alignat}
\end{subequations}
with some unknown constant $C_\mathrm{rad}\in\R$. 
Here, $\AA u$ is a (nonlinear) second-order elliptic differential operator
with $A:\Omega\times\R^d\to\R^d$, $b:\Omega\times\R^d\to\R$, and $c:\Omega\times\R\to\R$,
understood in the weak sense, i.e., $\FE : H^1(\Omega) \to \widetilde H^{-1}(\Omega) = H^1(\Omega)^*$,
\begin{align}\label{eq:defFE}
  \ip{\AA u}{v}_\Omega := \int_\Omega \big( A(\nabla u)\cdot\nabla v + b(\nabla u)v + {c(u)}v \big) \,dx
  \quad\text{for all } u,v\in H^1(\Omega).
\end{align}
We suppose that $\AA$ is strongly semi-monotone and Lipschitz continuous, i.e., 
there exist $c_\FE, C_\FE>0$ such that, for all $v,w\in H^1(\Omega)$, it holds that
\begin{align*}
  c_\FE \norm{\nabla w-\nabla v}{L^2(\Omega)}^2 &\leq \ip{\AA w-\AA v}{w-v}_\Omega 
  \quad\text{and}\quad 
  \norm{\AA w-\AA v}{\H^{-1}(\Omega)} \leq C_\FE \norm{w-v}{H^1(\Omega)}.
\end{align*}
In the case $d=2$, we suppose that $\diam(\Omega)<1$ to ensure coercivity of the weakly-singular integral
operator $\slo$, defined in Section~\ref{sec:intop}.
It is proved, e.g., in~\cite{affkmp,cs1995} that the model problem~\eqref{eq:modell} then admits a unique solution
$(u,u^\mathrm{ext})$.
We refer to Section~\ref{section:spaces} for the definition of the involved function spaces.

\subsection{Contributions and outline}
To develop our ideas, we first formulate the Uzawa-type iteration on the continuous level. To this end, Section~\ref{section:continuous:Uzawa} recalls the functional analytic setting (Section~\ref{section:spaces}) as well as the Johnson-N\'ed\'elec formulation of the transmission problem~\eqref{eq:modell} (Section~\ref{sec:intop}). Then, Algorithm~\ref{algorithm:uzawa:continuous} formulates the Uzawa iteration and Proposition~\ref{prop:uzawa:continuous} proves linear convergence with respect to the number of Uzawa iterations.

Section~\ref{section:discrete:Uzawa} is the mathematical core of the manuscript. We discretize each step of the Uzawa iteration by conforming BEM resp.\ FEM with piecewise polynomials of order $p-1$ resp.\ $p$. Algorithm~\ref{algorithm:uzawa:ideal} formulates the outer Uzawa iteration for the discretized problem, and Theorem~\ref{prop:uzawa:ideal} proves linear convergence. The inner iteration with adaptive FEM (resp.\ adaptive BEM) is the topic of Section~\ref{section:adaptive}. In the spirit of~\cite{ckns}, we give an abstract analysis of an adaptive mesh-refining algorithm (Algorithm~\ref{algorithm:adaptive}) which also allows the inexact solution of the arising linear systems by means of the preconditioned conjugate gradient method (PCG). Proposition~\ref{theorem:adaptive} proves that the algorithm reaches any prescribed tolerance in finite computational time. Moreover, for properly chosen preconditioners, the number of CG iterations in each step of the adaptive algorithm is uniformly bounded (Remark~\ref{remark:redpcg}). We apply this adaptive algorithm in each step of the outer Uzawa iteration for the BEM part (Section~\ref{section:uzawa_step_i}) and for the FEM part (Section~\ref{section:uzawa_step_ii}), where we employ a weighted-residual error estimator for the BEM and the standard residual error estimator for the FEM. Theorem~\ref{prop:bem:estimator} resp.\ Theorem~\ref{prop:fem:estimator} prove that the number of adaptive mesh-refinement steps (Algorithm~\ref{algorithm:adaptive}) in each step of the discrete Uzawa iteration (Algorithm~\ref{algorithm:uzawa:ideal}) is generically uniformly bounded.

Numerical experiments in Section~\ref{sec:ex} give empirical evidence that the proposed algorithm does not only provide a \emph{linear} solution strategy for a possibly \emph{nonlinear} transmission problem~\eqref{eq:modell}, but also leads to optimal convergence rates with respect to the number of elements.
 
\subsection{General notation}
Throughout the results, we state all constants as well as their dependencies. To abbreviate the presentation in proofs, we write $A\lesssim B$ if $A \le cB$ with a constant $c>0$ which is clear from the  context. Morever, $A \simeq B$ abbreviates $A \lesssim B \lesssim A$.

\section{Continuous Uzawa iteration}
\label{section:continuous:Uzawa}

\subsection{Involved function spaces}
\label{section:spaces}%
For any measurable subset $\omega\subseteq\Omega$ resp.\ $\omega\subseteq\Gamma$, let $\norm{\cdot}\omega:=\norm{\cdot}{L^2(\omega)}$ denote the
$L^2(\omega)$ norm which is induced by the $L^2(\omega)$ scalar product $\ip\cdot\cdot_\omega$.
Let $H^1(\Omega)$ denote the usual Sobolev space on $\Omega$ with norm 
\begin{align*}
  \norm{\cdot}{H^1(\Omega)}^2 := \norm{\cdot}\Omega^2 + \norm{\nabla(\cdot)}\Omega^2.
\end{align*}
Let $H^{1/2}(\Gamma)$ denote the fractional Sobolev space on the boundary with norm
\begin{align*}
  \norm{\widehat u}{H^{1/2}(\Gamma)} := \inf \set{\norm{u}{H^1(\Omega)}}{u\in H^1(\Omega) \text{ with } \gamma_0 u =
\widehat u},
\end{align*}
where $\gamma_0 :H^1(\Omega) \to H^{1/2}(\Gamma)$ denotes the trace operator.

Let $\widetilde H^{-1}(\Omega) = ( H^1(\Omega) )^*$ resp.\ $H^{-1/2}(\Gamma) = (H^{1/2}(\Gamma))^*$ denote the dual spaces of $H^1(\Omega)$ resp.\ $H^{1/2}(\Gamma)$, where the duality pairings extend the $L^2$ scalar products and are hence denoted by $\ip\cdot\cdot_\Omega$ resp.\ $\ip\cdot\cdot_\Gamma$. For $\lambda\in L^2(\Omega)\subseteq \widetilde H^{-1}(\Omega)$ and $\phi\in L^2(\Gamma)\subseteq H^{-1/2}(\Gamma)$, we thus have
\begin{align*}
  \ip{\lambda}{v}_\Omega = \int_\Omega {\lambda}v
  \quad\text{as well as}\quad
  \ip{\phi}w_\Gamma = \int_\Gamma\phi{w}.
\end{align*}
The norms on $\widetilde H^{-1}(\Omega)$ and $H^{-1/2}(\Gamma)$ are defined by duality, i.e.,
\begin{align*}
  \norm{\lambda}{\widetilde H^{-1}(\Omega)} := \sup_{v\in H^1(\Omega)\backslash\{0\}}
  \frac{\ip{\lambda}v_\Omega}{\norm{v}{H^1(\Omega)}}, \quad 
  \norm{\phi}{H^{-1/2}(\Gamma)} := \sup_{w\in H^{1/2}(\Gamma)\backslash\{0\}}
  \frac{\ip{\phi}w_\Gamma}{\norm{w}{H^{1/2}(\Gamma)}}.
\end{align*}
Finally, let $\Riesz : H^1(\Omega) \to \widetilde H^{-1}(\Omega)$ denote the Riesz mapping, i.e.,
\begin{align*}
 \ip{\Riesz w}v_\Omega = \ip{\nabla w}{\nabla v}_\Omega + \ip{w}v_\Omega 
 \quad\text{for all } v,w\in H^1(\Omega).
\end{align*}

\subsection{Johnson--N\'ed\'elec formulation of model problem}
\label{sec:intop}%
With $G(\cdot)$ being the fundamental solution of the Laplace operator, we consider the boundary integral operators
\begin{align}
 \slo\phi(x) := \int_\Gamma G(x-y)\,\phi(y)\,d\Gamma(y)
 \quad\text{and}\quad
 \dlo v(x) := \int_\Gamma \partial_{\normal(y)}G(x-y)\,v(y)\,d\Gamma(y).
\end{align}
The single-layer integral operator $\slo : H^{-1/2+s}(\Gamma) \to H^{1/2+s}(\Gamma)$ is an isomorphism for all $-1/2\le
s\le 1/2$. Moreover, for $s=0$, it is even elliptic and symmetric, i.e., $\ip{\phi}{\slo\phi}_\Gamma\geq
c_\slo \norm{\phi}{H^{-1/2}(\Gamma)}^2$ and
$\ip{\psi}{\slo\phi}_\Gamma = \ip{\phi}{\slo\psi}_\Gamma$ for all $\phi,\psi\in H^{-1/2}(\Gamma)$. The double-layer integral operator $\dlo:H^{1/2+s}(\Gamma)\to H^{1/2+s}(\Gamma)$ is a bounded linear operator for all $-1/2\le s\le 1/2$.

A common way to solve~\eqref{eq:modell}, is to rewrite the solution of the exterior problem~\eqref{eq:modell:ext} with
the help of the representation formula. This usually leads to equations involving boundary integral operators. Different
methods are available which are equivalent on the continuous level, but lead to different discrete formulations; see, e.g.,~\cite{affkmp,costabel,coers,han90,johned}.

In this work, we consider the Johnson--N\'ed\'elec coupling~\cite{johned} with its variational formulation:
Find $(u,\phi) \in H^1(\Omega)\times H^{-1/2}(\Gamma)$ such that
\begin{subequations}\label{eq:jn}
\begin{align}
  \ip{\AA u}{v}_\Omega
  - \ip{\phi}{\gamma_0 v}_\Gamma &= \ip{f}v_\Omega + \ip{\phi_0}{\gamma_0 v}_\Gamma, \label{eq:jn:fem}\\
  \ip{\psi}{(\tfrac12-\dlo)\gamma_0 u}_\Gamma + \ip{\psi}{\slo\phi}_\Gamma  &= \ip{\psi}{(\tfrac12-\dlo)u_0}_\Gamma, \label{eq:jn:bem}
\end{align}
\end{subequations}
for all $(v,\psi)\in H^1(\Omega)\times H^{-1/2}(\Gamma)$. It is known that~\eqref{eq:jn} admits a unique solution $(u,\phi) \in H^1(\Omega)\times H^{-1/2}(\Gamma)$, and $\phi = \partial_{\normal}u^{\rm ext}$ is the normal derivative of the exterior solution $u^\mathrm{ext}$. 

With $\gamma_0':H^{-1/2}(\Gamma)\to\H^{-1}(\Omega)$ being the adjoint of the trace operator, the Johnson--N\'ed\'elec coupling can equivalently be reformulated as follows: Find $(u,\phi)\in H^1(\Omega)\times H^{-1/2}(\Gamma)$ such that
\begin{subequations}\label{eq:johnson-nedelec}
\begin{align}
\AA u &= f + \gamma_0'(\phi_0 + \phi),\\
\slo\phi &= (\dlo-1/2)(\gamma_0 u-u_0).
\end{align}
\end{subequations}
This operator formulation provides the starting point for the following Uzawa-type iterative solvers.

\begin{remark}\label{rem:jn:discrete}
If $\XX_\bullet\subseteq H^1(\Omega)$ and $\YY_\bullet\subseteq H^{-1/2}(\Gamma)$ are conforming subspaces, it is proved in~\cite{affkmp} that the discrete Johnson--N\'ed\'elec coupling 
\begin{subequations}\label{eq:jn:discrete}
\begin{align}
  \ip{\AA u_\bullet}{v_\bullet}_\Omega
  - \ip{\phi_\bullet}{\gamma_0 v_\bullet}_\Gamma &= \ip{f}{v_\bullet}_\Omega + \ip{\phi_0}{\gamma_0 v_\bullet}_\Gamma, \\
  \ip{\psi_\bullet}{\slo\phi_\bullet}_\Gamma  &= \ip{\psi_\bullet}{(\dlo-\tfrac12)(\gamma_0 u_\bullet - u_0)}_\Gamma
\end{align}
\end{subequations}
admits a unique solution $(u_\bullet,\phi_\bullet)\in\XX_\bullet\times\YY_\bullet$ provided that $1\in\YY_\bullet$ and $c_\AA>0$ is sufficiently large. 
While~\cite{affkmp} requires $c_\AA>1/4$, one can adapt~\cite{os2014} to see that $c_\AA > c_{\dlo}/4$ is sufficient to ensure that the left-hand defines a (nonlinear) discrete bijection, where $c_{\dlo}\in[1/2,1)$ denotes the contraction constant of $1/2+\dlo$ in the $\slo^{-1}$-induced norm.
We stress, however, that our Uzawa-type iteration does not involve any assumption on $c_\AA>0$.\qed
\end{remark}

\subsection{Continuous Uzawa iteration}
The starting point for our analysis is the following iterative solution of the Johnson--N\'ed\'elec formulation~\eqref{eq:johnson-nedelec}. We stress that each step of the algorithm requires only the solution of two linear equations.

\begin{algorithm}\label{algorithm:uzawa:continuous}
{\bfseries Input:} Let $\alpha > 0$ and $u^{(0)}\in H^1(\Omega)$.\\
{\bfseries Uzawa iteration:} For all $j=1,2,3,\dots$, iterate the following steps~{\rm[i]--[iii]}:
\begin{itemize}
\item[{\rm[i]}] Solve $\slo\phi^{(j)} = (\dlo-1/2)(\gamma_0 u^{(j-1)} - u_0)$ for $\phi^{(j)}\in H^{-1/2}(\Gamma)$.
\item[{\rm[ii]}] Solve $\RR w^{(j)} = f - \AA u^{(j-1)} + \gamma_0'(\phi_0+\phi^{(j)})$ for $w^{(j)}\in H^1(\Omega)$.
\item[{\rm[iii]}] Define $u^{(j)} := u^{(j-1)} + \alpha\,w^{(j)} \in H^1(\Omega)$.
\end{itemize}
{\bfseries Output:} Sequence $(u^{(j)},\phi^{(j)})_{j\in\N}$ in $H^1(\Omega)\times H^{-1/2}(\Gamma)$.
\end{algorithm}

\begin{proposition}\label{prop:uzawa:continuous}
Suppose that $\alpha>0$ is sufficiently small. Then, for an arbitrary initial guess $u^{(0)}\in H^1(\Omega)$, the sequence $(u^{(j)},\phi^{(j)})_{j\in\N}$ from Algorithm~\ref{algorithm:uzawa:continuous} converges linearly to the solution $(u,\phi) \in H^1(\Omega)\times H^{-1/2}(\Gamma)$ of the Johnson--N\'ed\'elec formulation~\eqref{eq:johnson-nedelec}, i.e., there exist $C>0$ and $0<q<1$ such that, for all $j\in\N$ and all $n\in\N_0$ it holds that
\begin{align}\label{eq:uzawa:continuous}
 C^{-1}\norm{\phi-\phi^{(j+n+1)}}{H^{-1/2}(\Gamma)}
 \le \norm{u-u^{(j+n)}}{H^1(\Omega)}
 &\le q^n\,\norm{u-u^{(j)}}{H^1(\Omega)}.
\end{align}
\end{proposition}

\begin{proof}
The proof is split into four steps.

{\bf Step 1.} Recall from~\eqref{eq:johnson-nedelec} that $\slo\phi = (\dlo-1/2)(\gamma_0 u - u_0)$. From the mapping properties of $\slo$, $\dlo$, and $\gamma_0$, it follows that
\begin{align*}
 \norm{\phi-\phi^{(j)}}{H^{-1/2}(\Gamma)}
 \simeq \norm{\slo\phi-\slo\phi^{(j)}}{H^{1/2}(\Gamma)}
 &= \norm{(\dlo-1/2)\gamma_0(u-u^{(j-1)})}{H^{1/2}(\Gamma)}
 \\&
 \lesssim \norm{u-u^{(j-1)}}{H^1(\Omega)}
\end{align*}
for all $j\in\N$. This proves the first estimate of~\eqref{eq:uzawa:continuous}.

{\bf Step 2.} To prove the second estimate of~\eqref{eq:uzawa:continuous}, recall  from~\eqref{eq:johnson-nedelec} that $f=\AA u - \gamma_0'(\phi_0+\phi)$. Hence,
\begin{align*}
 \RR w^{(j)} = f - \AA u^{(j-1)} + \gamma_0'(\phi_0 + \phi^{(j)})
 = \AA u - \AA u^{(j-1)} - \gamma_0'(\phi - \phi^{(j)}).
\end{align*}
With $\slo(\phi-\phi^{(j)}) = (\dlo-1/2)\gamma_0(u-u^{(j-1)})$ and $\LL := \AA - \gamma_0' \slo^{-1} (\dlo-1/2) \gamma_0$, it follows that
\begin{align*}
 \RR w^{(j)} = \LL u - \LL u^{(j-1)}.
\end{align*}
Together with $u^{(j)} = u^{(j-1)} + \alpha w^{(j)}$, this proves
\begin{align}\label{dp:continuous}
 u - u^{(j)} = u - u^{(j-1)} - \alpha w^{(j)} = (1 - \alpha \RR^{-1} \LL) u - (1 - \alpha \RR^{-1} \LL) u^{(j-1)}.
\end{align}
Hence, it only remains to prove that the operator $(1 - \alpha \RR^{-1} \LL):H^1(\Omega)\to H^1(\Omega)$ is a contraction.

{\bf Step~3.} We show that the
operator $\LL:H^1(\Omega)\to\H^{-1}(\Omega)$ is strongly monotone and Lipschitz continuous, i.e., there exist $c_\LL, C_\LL>0$ such that, for all $v,w\in H^1(\Omega)$,
\begin{align*}
 c_\LL \norm{w-v}{H^1(\Omega)}^2 \leq \ip{\LL w-\LL v}{w-v}_\Omega 
 \text{ as well as }
 \norm{\LL w-\LL v}{\H^{-1}(\Omega)} \le C_\LL \,\norm{w-v}{H^1(\Omega)}.
\end{align*}
First, Lipschitz continuity follows from Lipschitz continuity of $\FE$ and boundedness of the trace operator $\gamma_0$
and the boundary integral operators $\slo$, $\dlo$.
Second, using the hypersingular integral operator $\hyp u(\xx) := -\partial_{\normal(\xx)} \int_\Gamma
\partial_{\normal(\yy)} G(\xx-\yy) \,d\Gamma(\yy)$, one can prove the identity
\begin{align*}
  \stpo:=-\slo^{-1}(\dlo-\tfrac12) = \hyp + (\tfrac12-\dlo')\slo^{-1}(\tfrac12-\dlo);
\end{align*}
see, e.g.,~\cite[Section~6.6]{bigObook}. This so-called exterior Steklov-Poincar\'e
operator is elliptic~\cite[Lemma~4]{cs1995} in $H^{1/2}(\Gamma)$. 
Together with strong semi-monotonicity of $\FE$, we see
\begin{align*}
  c_\FE \norm{\nabla (u-v)}{\Omega}^2 + c \norm{\gamma_0 (u-v)}{H^{1/2}(\Gamma)}^2 \leq 
  \ip{\FE u-\FE v}{u-v}_\Omega + \ip{\stpo\gamma_0(u-v)}{\gamma_0(u-v)}_\Gamma
\end{align*}
for all $u,v\in H^1(\Omega)$. Since the left-hand side defines an equivalent norm on $H^1(\Omega)$, this proves strong
monotonicity of $\LL$.

{\bf Step~4.} We finally show that the operator $(1 - \alpha \RR^{-1} \LL)$ is a contraction:
Let $0<\alpha<2c_\FE/C_\FE^2$. Let $w,v\in H^1(\Omega)$. Then,
\begin{align*}
  &\norm{(1-\alpha\Riesz^{-1}\LL)w-(1-\alpha\Riesz^{-1}\LL)v}{H^1(\Omega)}^2 
  \\&\quad
  = \norm{w-v}{H^1(\Omega)}^2 - 2\alpha \ip{\Riesz(\Riesz^{-1}\LL w-\Riesz^{-1}\LL v)}{w-v}_{\Omega}
  + \alpha^2 \norm{\Riesz^{-1}\LL w-\Riesz^{-1}\LL v}{H^1(\Omega)}^2
  \\&\quad
  = \norm{w-v}{H^1(\Omega)}^2 - 2\alpha \ip{\LL w-\LL v}{w-v}_\Omega + \alpha^2 \norm{\LL w-\LL v}{\H^{-1}(\Omega)}^2
  \\&\quad
  \le \norm{w-v}{H^1(\Omega)}^2 \,\big(1-2\alpha c_\LL + \alpha^2 C_\LL^2\big).
\end{align*}
By choice of $\alpha$, it holds that $q := 1-2\alpha c_\LL + \alpha^2 C_\LL^2 < 1$. In particular, it follows from~\eqref{dp:continuous} that
$\norm{u - u^{(j)}}{H^1(\Omega)} \le q\,\norm{u - u^{(j-1)}}{H^1(\Omega)}$, and an induction argument concludes~\eqref{eq:uzawa:continuous}.
\end{proof}

\section{Discrete Uzawa iteration}
\label{section:discrete:Uzawa}
\def\qref{q_{\rm ref}}%
\def\Ci{C_{\rm[i]}}%
\def\Cii{C_{\rm[ii]}}%

\subsection{Triangulation and mesh-refinement}\label{sec:refine}
Throughout, we assume that $\TT_\bullet$ is a conforming triangulation of $\Omega$ into compact non-degenerate simplices (of dimension $d$). By $\TT_\bullet|_\Gamma$, we denote the induced conforming triangulation of $\Gamma$ into plane non-degenerate surface simplices (of dimension $d-1$). 
For a simplex $T\in\TT_\bullet$, let $|T|$ denote its $d$-dimensional volume and let $\diam(T)$ be the Euclidean diameter.
The triangulation $\TT_\bullet$ is $\sigma$-shape regular if 
\begin{align}\label{eq:shape_regular}
 \max_{T\in\TT_\bullet} \frac{\diam(T)}{|T|^{1/d}} \le \sigma < \infty.
\end{align}
Note that this implies $|T|^{1/d} \le \diam(T) \le \sigma\,|T|^{1/d}$ for all $T\in\TT_\bullet$. Note that $\sigma$-shape regularity of $\TT_\bullet$ implies also the $\sigma$-shape regularity of $\TT_\bullet|_\Gamma$ in the sense of $|E|^{1/(d-1)}\le \diam(E) \le \sigma\,|E|^{1/(d-1)}$ for all facets $E\in\TT_\bullet|_\Gamma$, where $|E|$ denotes the $(d-1)$-dimensional surface area.

We suppose a fixed refinement strategy $\refine(\cdot)$, where $\TT_\circ = \refine(\TT_\bullet,\MM_\bullet)$ is the coarsest refinement of $\TT_\bullet$ such that all marked elements $\MM_\bullet\subseteq\TT_\bullet$ have been refined, i.e., $\MM_\bullet\subseteq\TT_\bullet\backslash\TT_\circ$. We suppose that
\begin{itemize}
\item each element $T\in\TT_\bullet$ is the union of its sons, i.e., $T = \bigcup\set{T'\in\TT_\circ}{T'\subseteq T}$;
\item there exists $0<\qref<1$ such that $|T'| \le \qref\,|T|$ for all $T\in\TT_\bullet\backslash\TT_\circ$ and all $T'\in\TT_\circ$ with $T'\subsetneqq T$, i.e., sons are uniformly smaller than their fathers.
\end{itemize}
We write $\TT_\circ\in\refine(\TT_\bullet)$, if there exist $n\in\N_0$, triangulations $\TT_0,\dots,\TT_n$, and marked elements $\MM_j\subseteq\TT_j$ such that $\TT_\bullet = \TT_0$, $\TT_{j+1}=\refine(\TT_j,\MM_j)$ for all $j=0,\dots,n-1$, and $\TT_\circ=\TT_n$. In particular, it holds that $\TT_\bullet\in\refine(\TT_\bullet)$. Finally, we suppose that $\refine(\cdot)$ guarantees uniform $\sigma$-shape regularity, i.e., all $\TT_\circ\in\refine(\TT_\bullet)$ are $\sigma$-shape regular~\eqref{eq:shape_regular}, where $\sigma>0$ depends only on $\TT_\bullet$.

One possible choice for $\refine(\cdot)$ is newest vertex bisection~\cite{kpp,stevenson}, where $\qref = 2^{-1/d}$.

\subsection{Discrete Uzawa iteration}
The discrete Uzawa iteration approximates $\phi^{(j)}\approx\phi^{(j)}_j\in\PP^{p-1}(\TT_j|_\Gamma)$ 
and $u^{(j)}\approx u^{(j)}_j\in\SS^p(\TT_j)$, where $\TT_j \in \refine(\TT_{j-1})$ for all $j\in\N$.
To formulate the basic idea, 
\begin{itemize}
\item let $\phi_\star^{(j)}\in H^{-1/2}(\Gamma)$ solve $\slo\phi_\star^{(j)} = (\dlo-1/2)(\gamma_0 u^{(j-1)}_{j-1} - u_0)$,
\item let $w^{(j)}_\star\in H^1(\Omega)$ solve $\RR w^{(j)}_\star = f - \AA u^{(j-1)}_{j-1} + \gamma_0'(\phi_0+\phi^{(j)}_j)$;
\end{itemize}
see also Algorithm~\ref{algorithm:uzawa:continuous} with the corresponding definition of $\phi^{(j)}$ resp.\ $w^{(j)}$.
With this notation, a discrete discrete Uzawa iteration reads as follows, where the precise computation of $\phi_j^{(j)}$ in step~[i] and $w_j^{(j)}$ in step~[ii] is the topic of Section~\ref{section:adaptive}--\ref{section:uzawa_step_ii}.

\begin{algorithm}\label{algorithm:uzawa:ideal}
{\bfseries Input:} Parameter $\alpha > 0$, initial triangulation $\TT_0$, initial guess $u^{(0)}_0\in \XX_0$, constants $\Ci,\Cii>0$ and $0<\gamma<1$.\\
{\bfseries Discrete Uzawa iteration:} For all $j=1,2,3,\dots$, iterate the following steps~{\rm[i]--[iii]}:
\begin{itemize}
\item[{\rm[i]}] Determine $\TT^{\rm[i]}_{j} \in \refine(\TT_{j-1})$ as well as some $\phi_j^{(j)}\in\PP^{p-1}(\TT^{\rm[i]}_{j}|_\Gamma)$ such that
\begin{align}\label{eq1:uzawa:ideal}
 \norm{\phi_\star^{(j)} - \phi_j^{(j)}}{H^{-1/2}(\Gamma)} 
 \le \Ci \, \gamma^j.
\end{align}
\item[{\rm[ii]}] Determine a triangulation $\TT^{\rm[ii]}_{j} \in \refine(\TT_j^{\rm[i]})$ and some $w_j^{(j)}\in\SS^{p}(\TT^{\rm[ii]}_{j})$ such that 
\begin{align}\label{eq2:uzawa:ideal}
 \norm{w^{(j)}_\star - w^{(j)}_j}{H^1(\Omega)}
 \le\Cii\,\gamma^j
\end{align}
\item[{\rm[iii]}] Define $\TT_j := \TT^{\rm[ii]}_{j}$ and $u_j^{(j)}:=u_{j-1}^{(j-1)}+\alpha w_j^{(j)}\in\SS^p(\TT_j)$.
\end{itemize}
\end{algorithm}

The following theorem together with the realization of step~{\rm[i]} and step~{\rm[ii]} which are presented below, is the main result of the present work.
\begin{theorem}\label{prop:uzawa:ideal}
Suppose that $\alpha>0$ is sufficiently small in the sense of Proposition~\ref{prop:uzawa:continuous}. 
Let $\Ci,\Cii>0$ as well as $0<\gamma<1$. Let $\TT_0$ be an arbitrary triangulation of $\Omega$.
Then, for an arbitrary discrete initial guess $u^{(0)}_0\in \SS^p(\TT_0)$, the sequence $(u^{(j)}_j,\phi^{(j)}_j)_{j\in\N}$ from Algorithm~\ref{algorithm:uzawa:ideal} converges to the solution $(u,\phi) \in H^1(\Omega)\times H^{-1/2}(\Gamma)$ of the Johnson--N\'ed\'elec formulation~\eqref{eq:johnson-nedelec}, and it holds
\begin{align}\label{eq:uzawa:ideal}
 \norm{\phi-\phi^{(j)}_j}{H^{-1/2}(\Gamma)} + \norm{u-u^{(j)}_j}{H^1(\Omega)} \le C\,\kappa^j
 \quad\text{for all }j\in\N,
\end{align}
where $C>0$ and $0<\kappa<1$ depend only on $\alpha$, $\Ci$, $\Cii$, $\norm{u-u^{(0)}_0}{H^1(\Omega)}$ and $\gamma$. Moreover, if $0<q<1$ is the contraction constant from Proposition~\ref{prop:uzawa:continuous} and $\gamma>q$, then $\kappa=\gamma$.
\end{theorem}

\begin{proof}
Recall from~\eqref{eq:johnson-nedelec} that $f = \AA u - \gamma_0'(\phi_0+\phi)$. Therefore, it follows that
\begin{align*}
 \RR w_\star^{(j)} = f-\AA u_{j-1}^{(j-1)}+\gamma_0'(\phi_0+\phi_j^{(j)})
 &= \AA u - \AA u_{j-1}^{(j-1)} - \gamma_0'(\phi-\phi_j^{(j)})
 \\&
 = \AA u - \AA u_{j-1}^{(j-1)} - \gamma_0'(\phi-\phi_\star^{(j)}) - \gamma_0'(\phi_\star^{(j)}-\phi_j^{(j)}).
\end{align*}
With $\slo(\phi-\phi_\star^{(j)}) = (\dlo-1/2)\gamma_0(u-u_{j-1}^{(j-1)})$
and $\LL := \AA-\gamma_0'\slo^{-1}(\dlo-1/2)\gamma_0$, it holds that
\begin{align*}
 \AA u - \AA u_{j-1}^{(j-1)} - \gamma_0'(\phi-\phi_\star^{(j)})
 = \LL u - \LL u_{j-1}^{(j-1)}.
\end{align*}
The combination of these two observations yields
\begin{align*}
 \RR w_\star^{(j)} 
 = \LL u - \LL u_{j-1}^{(j-1)}  - \gamma_0'(\phi_\star^{(j)}-\phi_j^{(j)}).
\end{align*}
Together with $u_j^{(j)} = u_{j-1}^{(j-1)} + \alpha\,w_j^{(j)}$, we hence obtain that
\begin{align*}
 u - u_j^{(j)}
 &= u - u_{j-1}^{(j-1)} - \alpha w_\star^{(j)} + \alpha (w_\star^{(j)}-w_j^{(j)})
 \\&= (1-\alpha\RR^{-1}\LL)u - (1-\alpha\RR^{-1}\LL)u_{j-1}^{(j-1)}
 + \alpha\RR^{-1}\gamma_0'(\phi_\star^{(j)}-\phi_j^{(j)})
 + \alpha (w_\star^{(j)}-w_j^{(j)}).
\end{align*}
As in the proof of Proposition~\ref{prop:uzawa:continuous}, it holds that
\begin{align*}
 \norm{u - u_j^{(j)}}{H^1(\Omega)}
 \le q\,\norm{u-u_{j-1}^{(j-1)}}{H^1(\Omega)}
 + \alpha\norm{\RR^{-1}\gamma_0'(\phi_\star^{(j)}-\phi_j^{(j)})}{H^1(\Omega)}
 + \alpha\,\norm{w_\star^{(j)}-w_j^{(j)}}{H^1(\Omega)}.
\end{align*}
Since $\RR^{-1}$ is an isometry and $\norm{\gamma_0(\cdot)}{H^{1/2}(\Gamma)}\le\norm{\cdot}{H^1(\Omega)}$,
it follows that
$\|\RR^{-1}\gamma_0':H^{-1/2}(\Gamma)\to H^1(\Omega)\|\le1$ for the operator norm. Together with~\eqref{eq1:uzawa:ideal}--\eqref{eq2:uzawa:ideal}, we obtain that
\begin{eqnarray*}
 \norm{u - u_j^{(j)}}{H^1(\Omega)}
 &\le&
 q\,\norm{u-u_{j-1}^{(j-1)}}{H^1(\Omega)}
 + \alpha \norm{\phi_\star^{(j)}-\phi_{j}^{(j)}}{H^{-1/2}(\Gamma)}
 + \alpha\,\norm{w_\star^{(j)}-w_j^{(j)}}{H^1(\Omega)}
 \\&
 \le&
 q\,\norm{u-u_{j-1}^{(j-1)}}{H^1(\Omega)}
 + \alpha(\Ci+\Cii)\,\gamma^j.
\end{eqnarray*}
Let $1 > \kappa > \widetilde\kappa := \max\{q,\gamma\}$. Arguing by induction on $j$, we prove that
\begin{align}\label{eq:sum}
 \begin{split}
 \norm{u - u_j^{(j)}}{H^1(\Omega)}
 &\le q^j\,\norm{u-u_{0}^{(0)}}{H^1(\Omega)}
 + \alpha(\Ci+\Cii)\,\sum_{\ell=0}^{j-1}q^\ell\gamma^{j-\ell}
 \\&
 \le \widetilde\kappa^{\,j}\,\norm{u-u_{0}^{(0)}}{H^1(\Omega)}
 + \alpha(\Ci+\Cii)\,j\,\widetilde\kappa^{\,j}.
 \end{split}
\end{align}
Note that $j\,\widetilde\kappa^{\,j} \lesssim \kappa^j$, where the hidden constant depends only on $\widetilde\kappa$ and $\kappa$. This concludes convergence~\eqref{eq:uzawa:ideal} for the FEM part, i.e., $\norm{u - u_j^{(j)}}{H^1(\Omega)} \lesssim \kappa^j$.

If $q<\gamma$, it holds that
\begin{align*} 
 \sum_{\ell=0}^{j-1} q^\ell \gamma^{j-\ell}
 = \gamma^j \sum_{\ell=0}^{j-1} (q/\gamma)^\ell
 = \gamma^j \, \frac{1-(q/\gamma)^j}{1-(q/\gamma)}.
\end{align*}
Using this estimate in~\eqref{eq:sum}, we prove convergence~\eqref{eq:uzawa:ideal} for the FEM part with $\kappa = \gamma > q$.

The mapping properties of the boundary integral operators reveal that
\begin{align*}
 \norm{\phi-\phi^{(j)}_\star}{H^{-1/2}(\Gamma)}
 \simeq \norm{\slo\phi-\slo\phi^{(j)}_\star}{H^{1/2}(\Gamma)}
 &= \norm{(\dlo-1/2)\gamma_0(u-u^{(j-1)}_{j-1})}{H^{1/2}(\Gamma)}
 \\&\lesssim \norm{u-u^{(j-1)}_{j-1}}{H^1(\Omega)}.
\end{align*}
Together with~\eqref{eq1:uzawa:ideal}--\eqref{eq2:uzawa:ideal}, we obtain that
\begin{align*}
 \norm{\phi-\phi^{(j)}_j}{H^{-1/2}(\Gamma)}
 &\le \norm{\phi-\phi^{(j)}_\star}{H^{-1/2}(\Gamma)} + \norm{\phi^{(j)}_\star-\phi^{(j)}_j}{H^{-1/2}(\Gamma)}
 \\&
 \lesssim \norm{u-u^{(j-1)}_{j-1}}{H^1(\Omega)} + \norm{\phi^{(j)}_\star-\phi^{(j)}_j}{H^{-1/2}(\Gamma)}
 \\&
 \lesssim \kappa^{j-1} + \gamma^{j}
 \le 2\kappa^{-1}\kappa^j.
\end{align*}
Therefore, convergence~\eqref{eq:uzawa:ideal} for the BEM part follows from the above arguments.
\end{proof}


\def\HH{\mathcal{H}}%
\def\UU{\mathcal{U}}%
\def\edual#1#2{\langle\!\langle#1\,,#2\rangle\!\rangle}%
\def\Cstab{C_{\rm stab}}%
\def\Crel{C_{\rm rel}}%
\def\qred{q_{\rm red}}%
\def\qctr{q_{\rm ctr}}%

\def\kctr{\kappa_{\rm ctr}}%
\def\lctr{\lambda_{\rm ctr}}%

\def\Cstp{C_{\rm stop}}%

\def\cond{\mathrm{cond}}
\def\CPCG{C_\mathrm{PCG}}

\subsection{Adaptivity with inexact PCG solver}\label{section:adaptive}
We will realize step~[i] and step~[ii] of Algorithm~\ref{algorithm:uzawa:ideal} by adaptive mesh-refining strategies which also include the use of the preconditioned conjugate gradient method (PCG). 
To this end, we follow~\cite{axioms} and note that step~[i] and step~[ii] can be covered simultaneously within the following abstract framework. 

Let $\HH$ be a Hilbert space with scalar product $\edual\cdot\cdot$ and corresponding norm $\enorm\cdot$. For each triangulation $\TT'_\bullet$, let $\XX_\bullet$ be an associated discrete subspace of $\HH$. We suppose that $\TT'_\circ\in\refine(\TT'_\bullet)$ implies nestedness $\XX_\bullet\subseteq\XX_\circ$. 
For given $F\in\HH^*$, let $\psi\in\HH$ be the exact solution of $\edual{\psi}\chi = F(\chi)$ for all
$\chi\in\XX$.
For some fixed $\psi\in\HH$, let $\psi_\bullet^\star\in\XX_\bullet$ be the best approximation of $\psi$ in $\XX_\bullet$, i.e., $\psi_\bullet^\star$ solves
\begin{align}\label{eq:cg:weakform}
  \edual{\psi_\bullet^\star}{\chi_\bullet} = F(\chi_\bullet) \quad\text{for all } \chi_\bullet\in\XX_\bullet.
\end{align}
or equivalently
\begin{align}\label{eq:abstract:discrete}
 \edual{\psi-\psi_\bullet^\star}{\chi_\bullet} = 0
 \quad\text{for all }\chi_\bullet\in\XX_\bullet.
\end{align}
Note that this also yields the Pythagoras theorem
\begin{align}\label{eq:abstract:pythagoras}
 \enorm{\psi-\psi_\bullet}^2 = \enorm{\psi-\psi_\bullet^\star}^2 + \enorm{\psi_\bullet^\star-\psi_\bullet}^2
 \ge \enorm{\psi-\psi_\bullet^\star}^2
 \quad\text{for all }\psi_\bullet\in\XX_\bullet.
\end{align}
For each $\psi_\bullet\in\XX_\bullet$ and all $T\in\TT'_\bullet$, we suppose some refinement indicator $\varrho_\bullet(T,\psi_\bullet)\ge0$. We define the corresponding error estimator
\begin{align}\label{eq:abstract:estimator}
 \varrho_\bullet(\psi_\bullet) := \varrho_\bullet(\TT'_\bullet,\psi_\bullet),
 \quad\text{where}\quad
 \varrho_\bullet(\UU_\bullet,\psi_\bullet) := \bigg(\sum_{T\in\UU_\bullet}\varrho_\bullet(T,\psi_\bullet)^2\bigg)^{1/2}
 \text{ for all }\UU_\bullet\subseteq\TT'_\bullet.
\end{align}
We suppose that there are constants $\Cstab,\Crel>0$ and $0<\qred<1$ such that for all $\TT'_\bullet\in\refine(\TT'_0)$ and all $\TT'_\circ\in\refine(\TT'_\bullet)$ as well as all $\psi_\bullet\in\XX_\bullet$ and $\psi_\circ\in\XX_\circ$, the following properties~\eqref{axiom:stability}--\eqref{axiom:reliability} are satisfied:
\begin{enumerate}
\renewcommand{\theenumi}{A\arabic{enumi}}
\bf\item\label{axiom:stability}\rm%
{\it\bfseries Stability on non-refined elements:} $|\varrho_\circ(\TT'_\circ\cap\TT'_\bullet,\psi_\circ) - \varrho_\bullet(\TT'_\circ\cap\TT'_\bullet,\psi_\bullet)| \le \Cstab\,\enorm{\psi_\circ-\psi_\bullet}$.
\bf\item\label{axiom:reduction}\rm%
{\it\bfseries Reduction on refined elements:} $\varrho_\circ(\TT'_\circ\backslash\TT'_\bullet,\psi_\bullet)^2 \le \qred \varrho_\bullet(\TT'_\bullet\backslash\TT'_\circ,\psi_\bullet)^2$.
\bf\item\label{axiom:reliability}\rm%
{\it\bfseries Reliability for exact best approximation:} $\enorm{\psi-\psi_\bullet^\star} \le \Crel\,\varrho_\bullet(\psi_\bullet^\star)$.
\end{enumerate}

Let $\{\xi_1,\dots,\xi_N\}\subseteq\XX_\bullet$ denote a basis of $\XX_\bullet$. Then, ~\eqref{eq:cg:weakform} is
equivalent to solving 
\begin{align}\label{eq:cg:matrixform}
  S x^\star = b
  \quad\text{with }
  S = \big(\edual{\xi_k}{\xi_j}\big)_{j,k=1\dots,N}\in\R^{N\times N}_{\rm sym}
  \quad\text{and}\quad
  b = \big(F(\xi_j)\big)_{j=1,\dots,N}\in\R^N
\end{align}
in the sense that $\psi_\bullet^\star = \sum_{j=1}^N x_j^\star \xi_j$.
Since $S$ is symmetric and positive definite, we use PCG as inexact solver and
replace the exact solution $x^\star\in\R^N$ by some PCG iteration, see~\cite{saad03,matcomp}.
To this end, we consider
\begin{align}\label{eq:cg:precondform}
  P^{-1}S x^\star = P^{-1}b
\end{align}
instead of~\eqref{eq:cg:matrixform}, where $P\in\R^{N\times N}_{\rm sym}$ is a symmetric and positive definite matrix which is spectrally equivalent to $S$, i.e., 
\begin{align}\label{eq:cg:specequiv}
  c_P x^T Px \leq x^T S x \leq C_P x^T Px \quad\text{for all }x\in\R^N.
\end{align}
We suppose that the constants $c_P,C_P>0$ are independent of $\XX_\bullet$ and call $P$ an optimal preconditioner for
$S$. Then, 
\begin{align}\label{eq:condCG}
  \cond_2(P^{-1/2}SP^{-1/2}) \leq \CPCG,
\end{align}
where $\CPCG$ depends only on $c_P$ and $C_P$, but is independent of $\XX_\bullet$.
We refer to~\cite{xcn09,wuchen06,ffps,ffpsFEMBEMAS} for optimal preconditioners for FEM and BEM on locally refined meshes.

\begin{lemma}[{\cite[Section~11.3 and~11.5]{matcomp}}]\label{prop:cg}
Let $\psi_{\bullet,0}\in \XX_\bullet$ and $x^{(0)}\in\R^N$ with $\psi_{\bullet,0} = \sum_{j=1}^N x^{(0)}_j \xi_j$. For $k=1,\dots,N$, let $x^{(k)}\in\R^N$ be the approximate solution of~\eqref{eq:cg:precondform} after $k$ iterations of the PCG algorithm~\cite[Algorithm~11.5.1]{matcomp} with matrix $S\in\R^{N\times N}_{\rm sym}$, optimal preconditioner $P\in\R^{N\times N}_{\rm sym}$, initial guess $x^{(0)}\in\R^N$, and right-hand side $b\in\R^N$. Let $\psi_{\bullet,k} := \sum_{j=1}^N x^{(k)}_j \xi_j\in\XX_\bullet$ be the corresponding discrete function.
Then, 
\begin{align*}
  \enorm{\psi_\bullet^\star-\psi_{\bullet,k}} &\leq \Big(1-\CPCG^{-1}\Big)^{1/2}\enorm{\psi_\bullet^\star-\psi_{\bullet,k-1}}, \\
  \enorm{\psi_\bullet^\star-\psi_{\bullet,k}} &\leq 2\left( \frac{\sqrt{\CPCG}-1}{\sqrt{\CPCG}+1} \right)^k \enorm{\psi_\bullet^\star-\psi_{\bullet,0}}.
\end{align*}
In particular, given a tolerance $\tau>0$, there exists a constant $K\in\N$ such that
\begin{align*}
  \enorm{\psi_\bullet^\star-\psi_{\bullet,k}} 
  \leq \tau \, \enorm{\psi_\bullet^\star-\psi_{\bullet,0}}
  \quad\text{for all }K \le k \le N.
\end{align*}
The constant $K$ depends only on $C_P$, $c_P$ from~\eqref{eq:cg:specequiv} as well as on $\tau$, but is independent of $\XX_\bullet$.\qed
\end{lemma}

The proof of the following proposition follows the ideas of~\cite{ckns}, but (unlike~\cite{ckns}) allows that $\psi_\ell\approx\psi_\ell^\star$ results from the inexact solution of~\eqref{eq:abstract:discrete} with, e.g., PCG.

\def\lcg{\lambda}
\begin{lemma}\label{prop:inexact}
Let $0<\theta\le1$ and suppose that $\TT'_{\ell+1} \in \refine(\TT'_\ell)$ satisfies, for some $\psi_\ell\in\XX_\ell$,
the D\"orfler marking criterion %
\begin{align}\label{eq:abstract:doerfler}
 \theta\varrho_\ell(\psi_\ell)^2 \le \varrho_\ell(\TT'_\ell\backslash\TT'_{\ell+1},\psi_\ell)^2.
\end{align}
Let $0\leq \lcg<1$.
Then, there exist $0<\kctr,\qctr<1$ such that the following assertion holds: 
If $\psi_{\ell+1}\in\XX_{\ell+1}$ is close to the exact best approximation $\psi_{\ell+1}^\star\in\XX_{\ell+1}$ in the sense of
\begin{align}\label{eq:abstract:inexact}
  \enorm{\psi_{\ell+1}^\star - \psi_{\ell+1}}^2 \leq \lambda \,\enorm{\psi_{\ell+1}^\star-\psi_\ell}^2,
\end{align}
then it follows that the so-called quasi-error is contractive, i.e.,
\begin{align}\label{eq:abstract:contraction}
 \Delta_{\ell+1} \le\qctr\,\Delta_{\ell},
 \quad\text{where}\quad
 \Delta_\bullet := \enorm{\psi-\psi_\bullet}^2 + \kctr\,\varrho_\bullet(\psi_\bullet)^2
\end{align}
The constants $\kctr$, $\qctr$ depend only on $\qred$, $\lcg$, $\Cstab$, $\Crel$, and $\theta$.
\end{lemma}

\begin{proof}
Applying the Pythagoras theorem~\eqref{eq:abstract:pythagoras} twice and using~\eqref{eq:abstract:inexact}, we prove that
\begin{eqnarray*}
 \enorm{\psi-\psi_{\ell+1}}^2
 &=& \enorm{\psi-\psi_{\ell+1}^\star}^2 + \enorm{\psi_{\ell+1}^\star-\psi_{\ell+1}}^2
 \\&
 =& \enorm{\psi-\psi_{\ell}}^2 - \enorm{\psi_{\ell+1}^\star-\psi_\ell}^2 + \enorm{\psi_{\ell+1}^\star-\psi_{\ell+1}}^2
  \\
  &\leq & \enorm{\psi-\psi_{\ell}}^2 - (1-\lcg)\enorm{\psi_{\ell+1}^\star-\psi_\ell}^2.
\end{eqnarray*}
In addition, we may also employ reliability~\eqref{axiom:reliability} and stability~\eqref{axiom:stability} to prove that
\begin{align}\label{eq0:abstract}
\begin{split}
 \enorm{\psi-\psi_{\ell+1}}^2
 &\,= \,\enorm{\psi-\psi_{\ell+1}^\star}^2 + \enorm{\psi_{\ell+1}^\star-\psi_{\ell+1}}^2
 \\
 &\reff{axiom:reliability}\le \Crel^2\,\varrho_{\ell+1}(\psi_{\ell+1}^\star)^2 + \enorm{\psi_{\ell+1}^\star-\psi_{\ell+1}}^2
 \\
 &\reff{axiom:stability}\le 2\Crel^2\,\varrho_{\ell+1}(\psi_{\ell+1})^2 + (1+2\Crel^2\Cstab^2)\enorm{\psi_{\ell+1}^\star-\psi_{\ell+1}}^2
  \\
 & \le 2\Crel^2\,\varrho_{\ell+1}(\psi_{\ell+1})^2 + (1+2\Crel^2\Cstab^2)\lcg\enorm{\psi_{\ell+1}^\star-\psi_{\ell}}^2.
\end{split}
\end{align}
Let $\delta>0$ which will be fixed later. Define $C'' := (1+2\,\Crel^2\Cstab^2)$. Then, the last two estimates lead to
\begin{eqnarray}\label{eq1:abstract}
 \notag
 \enorm{\psi-\psi_{\ell+1}}^2 
 &=& 
 (1-\delta)\,\enorm{\psi-\psi_{\ell+1}}^2
 + \delta\,\enorm{\psi-\psi_{\ell+1}}^2
 \\
 &\le& 
 \notag
  (1-\delta)\,\enorm{\psi-\psi_{\ell}}^2 - \big( (1-\delta)(1-\lcg) - C''\delta\lcg \big)\,\enorm{\psi_{\ell+1}^\star-\psi_\ell}^2 
 \\&&\quad
 + 2\delta\Crel^2\,\varrho_{\ell+1}(\psi_{\ell+1})^2.
\end{eqnarray}
Having bounded the energy error, we consider the estimator. For all $\eps>0$, it holds that
\begin{eqnarray*}
 \varrho_{\ell+1}(\psi_{\ell+1})^2
 &\reff{axiom:stability}\le& 
 \big(\varrho_{\ell+1}(\psi_{\ell}) + \Cstab\,\enorm{\psi_{\ell+1}-\psi_\ell}\big)^2
 \\
 &\le& (1+\eps)\,\varrho_{\ell+1}(\psi_{\ell})^2 + (1+\eps^{-1})\,\Cstab^2\,\enorm{\psi_{\ell+1}-\psi_\ell}^2
 \\
 &\le& (1+\eps)\,\varrho_{\ell+1}(\psi_{\ell})^2 + 2(1+\eps^{-1})\,\Cstab^2\,\big(\enorm{\psi_{\ell+1}^\star-\psi_\ell}^2
 +\enorm{\psi_{\ell+1}^\star-\psi_{\ell+1}}^2\big)
 \\
 &\reff{eq:abstract:inexact}\le&
 (1+\eps)\,\varrho_{\ell+1}(\psi_{\ell})^2 
 + 2(1+\eps^{-1})\,\Cstab^2(1+\lcg) \, \enorm{\psi_{\ell+1}^\star-\psi_\ell}^2.
\end{eqnarray*}
Moreover, stability~\eqref{axiom:stability}, reduction~\eqref{axiom:reduction}, and D\"orfler marking~\eqref{eq:abstract:doerfler} yield that
\begin{eqnarray*}
 \varrho_{\ell+1}(\psi_{\ell})^2
 &=& \varrho_{\ell+1}(\TT'_{\ell+1}\cap\TT'_\ell,\psi_{\ell})^2 + \varrho_{\ell+1}(\TT'_{\ell+1}\backslash\TT'_\ell,\psi_{\ell})^2
 \\
 &\le& \varrho_\ell(\TT'_{\ell+1}\cap\TT'_\ell,\psi_{\ell})^2 + \qred\,\varrho_\ell(\TT'_\ell\backslash\TT'_{\ell+1},\psi_{\ell})^2
 \\
 &=& \varrho_{\ell}(\psi_{\ell})^2 - (1-\qred)\,\varrho_\ell(\TT'_\ell\backslash\TT'_{\ell+1},\psi_{\ell})^2
 \\
 &\reff{eq:abstract:doerfler}\le&
 \big(1-(1-\qred)\theta\big)\,\varrho_{\ell}(\psi_{\ell})^2.
\end{eqnarray*}
We set $q':= 1-(1-\qred)\theta < 1$ and $C':=2\Cstab^2(1+\lcg)$.
Combining the last two estimates, we obtain that
\begin{align}\label{eq2:abstract}
 \varrho_{\ell+1}(\psi_{\ell+1})^2
  \le (1+\eps)q'\,\varrho_{\ell}(\psi_{\ell})^2 + (1+\eps^{-1})C'\,\enorm{\psi_{\ell+1}^\star-\psi_\ell}^2.
\end{align}
Let $\kctr>0$ which is fixed later. Combining~\eqref{eq1:abstract}--\eqref{eq2:abstract}, we infer that
\begin{eqnarray*}
 \Delta_{\ell+1} 
 &=& \enorm{\psi-\psi_{\ell+1}}^2 + \kctr\,\varrho_{\ell+1}(\psi_{\ell+1})^2
 \\&
 \reff{eq1:abstract}\le& (1-\delta)\,\enorm{\psi-\psi_{\ell}}^2 
  - \big( (1-\delta)(1-\lcg) - C''\delta\lcg \big)\enorm{\psi_{\ell+1}^\star-\psi_\ell}^2
 \\&&\quad
 + (\kctr+2\delta\Crel^2)\,\varrho_{\ell+1}(\psi_{\ell+1})^2
 \\&
 \reff{eq2:abstract}\le& (1-\delta)\,\enorm{\psi-\psi_{\ell}}^2
  + (\kctr+2\delta\Crel^2)(1+\eps)q'\,\varrho_\ell(\psi_\ell)^2
 \\&&\quad
   - \big\{(1-\delta)(1-\lcg) - C''\delta\lcg - (\kctr+2\delta\Crel^2) C' (1+\eps^{-1})\big\}\,\enorm{\psi_{\ell+1}^\star-\psi_\ell}^2
 \\
 &=&
 (1-\delta)\,\enorm{\psi-\psi_{\ell}}^2
  + \kctr\, (1+ \kctr^{-1}2\delta\Crel^2)(1+\eps)q'\,\varrho_\ell(\psi_\ell)^2
 \\&&\quad
   - \big\{(1-\delta)(1-\lcg) - C''\delta\lcg - (\kctr+2\delta\Crel^2) C' (1+\eps^{-1})\big\}\,\enorm{\psi_{\ell+1}^\star-\psi_\ell}^2.
\end{eqnarray*}
It remains to choose $\delta,\kctr,\eps>0$. 
First, choose $\eps>0$ sufficiently small such that $(1+\eps)q'<1$.
Then, choose $\kctr>0$ sufficiently small such that $(1-\lcg)-\kctr C'(1+\eps^{-1})>0$.
Finally, choose $\delta>0$ such that
\begin{itemize}
\item $q'' := \big\{ (1+\kctr^{-1}2\delta\Crel^2)(1+\eps)q' \big\} < 1$,
\item $\big\{(1-\delta)(1-\lcg) - C''\delta\lcg - (\kctr+2\delta\Crel^2) C' (1+\eps^{-1})\big\}\ge0$.
\end{itemize}
This leads to
\begin{align*}
 \Delta_{\ell+1} \le (1-\delta)\,\enorm{\psi-\psi_{\ell}}^2 + \kctr q''\,\varrho_\ell(\psi_\ell)^2
 \le \max\{1-\delta\,,\,q''\} \, \Delta_\ell.
\end{align*}
and hence concludes the proof.
\end{proof}

\begin{remark}\label{remark:reliable}
Note that~\eqref{eq0:abstract} shows that
\begin{align}\label{eq:remark:reliable}
  \enorm{\psi - \psi_\ell}^2 \lesssim \rho_\ell(\psi_\ell)^2 + \enorm{\psi_\ell^\star-\psi_\ell}^2,
\end{align}
where the hidden constants depend only on $\Crel,\Cstab>0$. The error term $\enorm{\psi_\ell^\star-\psi_\ell}^2$ can 
efficiently be evaluated in an equivalent norm: Let $x^\star,x\in\R^N$ be the coefficent vectors of
$\psi_\ell^\star$ resp.\, $\psi_\ell$, $S$ be the stiffness matrix of $\edual\cdot\cdot$, and $P$ be an optimal preconditioner.
Then,
\begin{align}
\begin{split}\label{eq:cg:normequiv}
  \enorm{\psi_\ell^\star-\psi_\ell}^2 &= (x^\star-x)^T S (x^\star-x) = \big( b-Sx\big)^T S^{-1} \big(b-Sx\big) \\
  &\simeq \big( b-Sx\big)^T P^{-1} \big(b-Sx\big) =: \enorm{\psi_\ell^\star-\psi_\ell}_P^2, 
\end{split}
\end{align}
where the hidden constants depend only on $c_P,C_P$ from~\eqref{eq:cg:specequiv}. Note that $\enorm{\psi_\ell^\star-\psi_\ell}_P^2$ is evaluated in
each iteration of the PCG algorithm; see~\cite[Algorithm~11.5.1]{matcomp}. 
Therefore, no extra computational cost is needed.\qed
\end{remark}

\begin{algorithm}\label{algorithm:adaptive}
\textbf{Input:} Parameter $0<\theta\le1$ as well as $0\le\lambda<1$, initial triangulation $\TT'_0$, initial guess
  $\psi_{-1}\in\XX_0$, as well as tolerance $\tau>0$.\\
\textbf{Adaptive loop:} For all $\ell=0,1,2,\dots$, iterate the following steps~{\rm(i)--(iii)}, until
\begin{align}\label{eq:algorithm:stopping}
  \varrho_{\ell}(\psi_{\ell})^2 + \enorm{\psi_\ell^\star-\psi_\ell}_P^2 \le\tau^2,
\end{align}
where $\enorm{\cdot}_P$ is defined in~\eqref{eq:cg:normequiv}:
\begin{itemize}
\item[\rm(i)] Compute an approximate solution $\psi_\ell:=\psi_{\ell,k}\in\XX_\ell$
to~\eqref{eq:cg:weakform}, where $k\in\N_0$ is the minimal number such that 
the $k$-th iterate $\psi_{\ell,k}$ in PCG with initial guess $\psi_{\ell,0}:=\psi_{\ell-1}$ (see Proposition~\ref{prop:cg}) satisfies
\begin{align}\label{eq:algorithm:inexact}
  \enorm{\psi_\ell^\star-\psi_\ell}^2 \leq \lambda \, \enorm{\psi_\ell^\star-\psi_{\ell-1}}^2.
\end{align}
\item[\rm(ii)] Determine a set of marked elements $\MM'_\ell\subseteq\TT'_\ell$ such that
\begin{align}\label{eq:algorithm:doerfler}
 \theta\,\varrho_\ell(\psi_\ell)^2 \le \varrho_\ell(\MM'_\ell,\psi_\ell)^2.
\end{align}
\item[\rm(iii)] Generate new triangulation $\TT'_{\ell+1}:=\refine(\TT'_\ell,\MM'_\ell)$.
\end{itemize}
{\bfseries Output:} Smallest index $\ell$, adaptively refined triangulation $\TT'_\ell\in\refine(\TT'_0)$, 
and discrete approximation $\psi_\ell\in\XX_\ell$ which satisfies the stopping criterion~\eqref{eq:algorithm:stopping}.
\end{algorithm}

\begin{remark}\label{remark:redpcg}
Proposition~\ref{prop:cg} proves that for fixed $0\leq\lambda<1$, the smallest number of PCG iterations $k$ such
that~\eqref{eq:algorithm:inexact} holds, is uniformly bounded by some $K\in\N$ that depends only on $\lambda$ 
and $\CPCG$, but not on $\ell\in\N_0$. For the particular choice $\lambda := (1-\CPCG^{-1})^{1/2}$, the
condition~\eqref{eq:algorithm:inexact} isa already satisfied after one PCG step.
\end{remark}

\begin{proposition}\label{theorem:adaptive}
Let $0<\theta\le1$, $0\leq \lcg<1$ and $\tau>0$.
Then, Algorithm~\ref{algorithm:adaptive} terminates after finitely many iterations and provides some triangulation 
$\TT'_\ell\in\refine(\TT'_0)$ together with some discrete approximation $\psi_\ell\in\XX_\ell$ to $\psi\in\HH$ such that
\begin{align}
 \enorm{\psi-\psi_\ell} \le \Cstp\,\tau,
\end{align}
where $\Cstp>0$ depends only on $\Crel$ and $\Cstab$.
\end{proposition}

\begin{proof}
Since marked elements are refined, i.e., $\MM'_\ell\subseteq\TT'_\ell\backslash\TT'_{\ell+1}$, the marking
criterion~\eqref{eq:algorithm:doerfler} ensures that~\eqref{eq:abstract:doerfler} is satisfied in each step $\ell\ge0$ of the adaptive loop. For $\ell\ge1$, the accuracy criterion~\eqref{eq:algorithm:inexact} coincides with~\eqref{eq:abstract:inexact}. Hence, Proposition~\ref{prop:inexact} applies and provides $0<\kctr,\qctr<1$ with~\eqref{eq:abstract:contraction}. In particular, this guarantees
\begin{align*}
 \enorm{\psi-\psi_\ell}^2 + \kctr\,\varrho_{\ell}(\psi_\ell)^2 
 \le \qctr^\ell\,\big(\enorm{\psi-\psi_0}^2 + \kctr\,\varrho_{\ell}(\psi_0)^2\big)
 \quad\text{for all }\ell\ge0.
\end{align*}
In particular (and formally for $\tau=0$), this proves $\varrho_{\ell}(\psi_\ell)^2 + \enorm{\psi-\psi_\ell}^2\to0$ as $\ell\to\infty$.
Since $\enorm{\psi_\ell^\star-\psi_\ell}_P \simeq \enorm{\psi_\ell^\star-\psi_\ell}\leq \enorm{\psi-\psi_\ell}$, this also shows 
$\enorm{\psi_\ell^\star-\psi_\ell}_P \to 0$.
Hence, there exists a minimal $\ell\in\N_0$ such that the stopping criterion~\eqref{eq:algorithm:stopping} is satisfied and Algorithm~\ref{algorithm:adaptive} terminates.
Remark~\ref{remark:reliable} provides some constant $\Cstp>0$ which depends only on $\Crel$ and $\Cstab$, such that
\begin{align*}
  \enorm{\psi-\psi_\ell}^2 \reff{eq:remark:reliable}\le \Cstp (\varrho_\ell(\psi_\ell)^2 + \enorm{\psi_\ell^\star-\psi_\ell}_P^2)
  \reff{eq:algorithm:stopping}\le \Cstp\, \tau.
\end{align*}
This concludes the proof.
\end{proof}


\def\Cinit{C_{\rm init}}
\def\Capx{C_{\rm apx}}

\subsection{Realization of step~[i] of Uzawa iteration}
\label{section:uzawa_step_i}%
Step~[i] of Algorithm~\ref{algorithm:uzawa:ideal} will be realized by means of Algorithm~\ref{algorithm:adaptive}, where 
$$
\psi := \phi_\star^{(j)} \in H^{-1/2}(\Gamma)=:\HH,
\quad 
\XX_\bullet := \PP^{p-1}(\TT_\bullet|_\Gamma), 
\quad\text{and}\quad 
\edual\cdot\cdot:=\ip{V(\cdot)}{(\cdot)}_\Gamma.
$$
We employ the weighted-residual error estimator from~\cite{cs95,cc97,cms01}. We note, however, that the residual involves the integration of $(\dlo-1/2)u_0$ which can hardly be performed for continuous data $u_0\in H^{1/2}(\Gamma)$. Therefore, we follow~\cite{partOne}, suppose additional regularity $u_0\in H^1(\Gamma)$, and approximate $u_0\approx u_{0,\bullet}\in\SS^1(\TT_\bullet|_\Gamma)$. This additional approximation error is also included in the {\sl a~posteriori} error estimator. 
Let $\nabla_\Gamma(\cdot)$ denote the surface gradient.
Recall that $V\phi_\star^{(j)} = (\dlo-1/2)(\gamma_0 u_{j-1}^{(j-1)}-u_0)$. Then, the overall estimator reads
\begin{align*}
 \mu_\bullet(E,\psi_\bullet)^2
 &:=
 |E|^{1/(d-1)}\,\norm{\nabla_\Gamma\big((\dlo-1/2)(\gamma_0u_{j-1}^{(j-1)}-u_{0,\bullet})-\slo\psi_\bullet\big)}{L^2(E)}^2
 \\&\qquad
 + |E|^{1/(d-1)}\,\norm{(1-\Pi_\bullet)\nabla_\Gamma u_0}{L^2(E)}^2,
\end{align*}
where $\Pi_\bullet:L^2(\Gamma)\to\PP^{p-1}(\TT_\bullet|_\Gamma)$ denotes the $L^2$-orthogonal projection onto $\PP^{p-1}(\TT_\bullet|_\Gamma)$. 

\begin{lemma}[{\cite[Proposition~2]{partOne} and \cite[Section~6]{partOne}}]
\label{lemma:bem:estimator}
Suppose that the discretization $u_{0,\bullet}\in\SS^p(\TT_\bullet|_\Gamma)$ of $u_0\in H^1(\Gamma)$ is obtained
\begin{itemize}
\item either by the Scott-Zhang projection~\cite{partOne} onto $\SS^p(\TT_\bullet|_\Gamma)$ for $p\ge1$ and $d\ge2$,
\item or by the $L^2$-orthogonal projection onto $\SS^p(\TT_\bullet|_\Gamma)$ for $p\ge1$ and $d=2$,
\item or by the $L^2$-orthogonal projection onto $\SS^p(\TT_\bullet|_\Gamma)$ for $p\ge1$ and $d=3$, if this is $H^1$-stable (see~\cite{kpp,ghs2016}),
\item or by nodal interpolation for $p=1$ and $d=2$.
\end{itemize}
Suppose that $\refine(\cdot)$ releies on newest vertex bisection~\cite{kpp,stevenson} for $d=3$.
Then, the error estimator $\mu_\bullet(\cdot)$ satisfies the assumptions~\eqref{axiom:stability}--\eqref{axiom:reliability} from Section~\ref{section:adaptive}, where $\Cstab$, $\Crel$, and $0<\qred<1$ depend only on the mesh-refinement strategy $\refine(\cdot)$ and $\sigma$-shape regularity of $\TT_0$. Moreover, in all these cases, $\TT_\circ\in\refine(\TT_\bullet)$ implies that 
\begin{align}\label{eq:lemma:bem:estimator}
 \Capx^{-1}\,\norm{u_{0,\circ}-u_{0,\bullet}}{H^{1/2}(\Gamma)}
 \le \bigg(\sum_{E\in\TT_\bullet|_\Gamma}|E|^{1/(d-1)}\,\norm{(1-\Pi_\bullet)\nabla_\Gamma u_0}{L^2(E)}^2\bigg)^{1/2}
 \le \mu_\bullet(\psi_\bullet)
\end{align}
for all $\psi_\bullet\in\PP^{p-1}(\TT_\bullet|_\Gamma)$,
where $\Capx>0$ depends only on $\Gamma$, $p$, and $\sigma$-shape regularity of $\TT_\bullet$ as well as the use of newest vertex bisection for $d=3$.
\qed
\end{lemma}

Because of Lemma~\ref{lemma:bem:estimator}, we can employ Algorithm~\ref{algorithm:adaptive} to realize step~[i] of Algorithm~\ref{algorithm:uzawa:ideal}. In particular, the following theorem proves that the number of adaptive iterations of Algorithm~\ref{algorithm:adaptive} is uniformly bounded.

\begin{theorem}\label{prop:bem:estimator}
Let $0<\gamma<1$, $0<\theta\le1$, and $0\leq\lcg<1$.
For $j\in\N$, choose $\tau = \gamma^j$ and $\TT'_0:=\TT_{j-1}$. Then, the following assertions {\rm(a)--(b)} hold:

{\rm(a)}
After $\ell\in\N_0$ iterations, Algorithm~\ref{algorithm:adaptive} returns the triangulation $\TT_j^{\rm[i]} = \TT'_\ell
\in \refine(\TT_j)$ and a corresponding discrete function $\phi_j^{(j)}\in\PP^{p-1}(\TT_\ell'|_\Gamma)$ such that 
\begin{align}\label{eq:step1:termination} 
 \norm{\phi^{(j)}_\star-\phi_j^{(j)}}{H^{-1/2}(\Gamma)}\le\Ci\,\gamma^j,
\end{align}
where $\Ci>0$ depends only on $\Crel$, $\Cstab$, and $\Gamma$. 

{\rm(b)} Suppose that $\alpha>0$ is sufficiently small in the sense of Proposition~\ref{prop:uzawa:continuous} and that $0<q<1$ is the resulting contraction constant of the continuous Uzawa iteration. Suppose $q<\gamma$. 
Moreover, suppose that there exists $\Cinit>0$ such that for all $j\ge1$, the initial guess $\phi_{j,0}^{(j)} = \phi_{j-1}^{(j)} \in \PP^{p-1}(\TT_{j-1}|_\Gamma)$ for Algorithm~\ref{algorithm:adaptive} satisfies
\begin{align}\label{eq:step1:initial}
 \norm{\phi_{j-1}^{(j),\star} - \phi_{j-1}^{(j)}}{H^{-1/2}(\Gamma)}
 \le \Cinit\,\norm{\phi_{j-1}^{(j),\star} - \phi_{j-1}^{(j-1)}}{H^{-1/2}(\Gamma)},
\end{align}
where $\phi_{j-1}^{(j),\star} \in \PP^{p-1}(\TT_{j-1}|_\Gamma)$ is the best approximation of $\phi^{(j)}_\star$ in $\PP^{p-1}(\TT_{j-1}|_\Gamma)$ with respect to $\norm{\cdot}{H^{-1/2}(\Gamma)}$.
Then, the number $\ell\in\N_0$ of iterations in Algorithm~\ref{algorithm:adaptive} is uniformly bounded for all $j\in\N$, i.e., $\ell\le L$, where $L>0$ depends only on $\alpha$, $\Ci$, $\Cii$, $\Cinit$, $\gamma$, $p$ as well as on uniform $\sigma$-shape regularity of the triangulation $\TT_j\in\refine(\TT_0)$ and on $\Gamma$.
\end{theorem}

\begin{remark}
Note that~\eqref{eq:step1:initial} allows the choice $\phi_{j,0}^{(j)} = \phi_{j-1}^{(j-1)}$. However, in our implementation, we obtain $\phi_{j,0}^{(j)}$ by one CG iteration with initial value $\phi_{j-1}^{(j-1)}$. Then, $\Cinit$ depends only on the norm equivalence $\enorm\cdot \simeq \norm{\cdot}{H^{-1/2}(\Gamma)}$ and hence on $\Gamma$.\qed
\end{remark}

\begin{proof}[Proof of Theorem~\ref{prop:bem:estimator}]
To prove~(a), note the norm equivalence $\enorm{\chi} :=
\ip{V\chi}{\chi}_\Gamma^{1/2}\simeq\norm{\chi}{H^{-1/2}(\Gamma)}$ for all $\chi\in H^{-1/2}(\Gamma)$. The first claim
together with the estimate~\eqref{eq:step1:termination} follows from Proposition~\ref{theorem:adaptive}, where $\Ci \simeq \Cstp$ with hidden norm equivalence constants.

To prove~(b), note that the number of iterations is finite for $j=1$. Without loss of generality, we may hence
suppose $j\ge2$. Let $\TT_{j,\ell} = \TT_\ell'$ be the $\ell$-th adaptive mesh in Algorithm~\ref{algorithm:adaptive} (in the $j$-th iteration of Algorithm~\ref{algorithm:uzawa:ideal}). 
Let $\phi^{(j)}_{j,\ell} \in \PP^{p-1}(\TT_{j,\ell}|_\Gamma)$ be the corresponding approximation.
Recall that Algorithm~\ref{algorithm:adaptive} guarantees that
\begin{align*}
 \enorm{\phi^{(j)}_\star-\phi^{(j)}_{j,\ell}}^2 + \kctr\mu_{j,\ell}(\phi^{(j)}_{j,\ell})^2
 \le \qctr^\ell \,
 \big(\enorm{\phi^{(j)}_\star-\phi^{(j)}_{j,0}}^2 + \kctr\mu_{j,0}(\phi^{(j)}_{j,0})^2\big)
 \quad\text{for all }\ell\ge0.
\end{align*} 
This proves that
\begin{align*}
 \mu_{j,\ell}(\phi^{(j)}_{j,\ell})^2
 \le \kctr^{-1}\qctr^\ell\,\big(\enorm{\phi^{(j)}_\star-\phi^{(j)}_{j,0}}^2 + \kctr\mu_{j,0}(\phi^{(j)}_{j,0})\big)^2
 \quad\text{for all }\ell\ge0.
\end{align*}
To conclude the proof of~(b), it only remains to show that 
\begin{align}\label{eq0:prop:bem:estimator}
 \enorm{\phi^{(j)}_\star-\phi^{(j)}_{j,0}} + \mu_{j,0}(\phi^{(j)}_{j,0}),
 \le C'\,\gamma^{j-1}
\end{align}
where $C'>0$ is independent of $j$. For sufficiently large $\ell=L$ (which does \emph{not} depend on $j$) and $\tau = \gamma^j$, the stopping criterion~\eqref{eq:algorithm:stopping} is then satisfied and hence Algorithm~\ref{algorithm:adaptive} terminates for some $\ell\le L$.

For the ease of presentation, we suppose $\enorm\cdot = \norm{\cdot}{H^{-1/2}(\Gamma)}$ so that all estimates hold up to norm equivalence constants (which, however, depend only on $\Gamma$).
The proof of~\eqref{eq0:prop:bem:estimator} is split into several steps.

{\bf Step~1.} Recall that $\TT_{j,0}=\TT_{j-1} = \TT_{j-1}^{\rm[ii]}\in\refine(\TT_{j-1}^{\rm[i]})$. 
For $k\in\{j-1,j\}$, let $\phi^{(k),\star}_{j,0}\in\PP^{p-1}(\TT_{j,0}|_\Gamma)$ be the best approximation of $\phi^{(k)}_\star$  in $\PP^{p-1}(\TT_{j,0}|_\Gamma)$ with respect to $\norm{\cdot}{H^{-1/2}(\Gamma)}$. Then, the triangle inequality and elementary properties of the orthogonal projection prove that
\begin{align*}
 \enorm{\phi_\star^{(j)} \!-\! \phi_{j,0}^{(j)}}
 &\le \enorm{\phi_\star^{(j)} \!-\! \phi_\star^{(j-1)}}
 + \enorm{\phi_\star^{(j-1)} \!-\! \phi_{j,0}^{(j-1),\star}}
 + \enorm{\phi_{j,0}^{(j-1),\star} \!-\! \phi_{j,0}^{(j),\star} }
 + \enorm{\phi_{j,0}^{(j),\star} \!-\! \phi_{j,0}^{(j)}}
 \\&
 \le 2\,\enorm{\phi_\star^{(j)} \!-\! \phi_\star^{(j-1)}}
 + \enorm{\phi_\star^{(j-1)} \!-\! \phi_{j-1}^{(j-1)}}
 + \enorm{\phi_{j,0}^{(j),\star} \!-\! \phi_{j,0}^{(j)}}
 \\&
 \reff{eq:step1:initial}\le 2\,\enorm{\phi_\star^{(j)} - \phi_\star^{(j-1)}} 
 + \enorm{\phi_\star^{(j-1)} - \phi_{j-1}^{(j-1)}}
 + \Cinit\,\enorm{\phi_{j,0}^{(j),\star} - \phi_{j-1}^{(j-1)}}
\end{align*}
and
\begin{align*}
 \enorm{\phi_{j,0}^{(j),\star} - \phi_{j-1}^{(j-1)}}
 &\le 
 \enorm{\phi_{j,0}^{(j),\star} - \phi_{j,0}^{(j-1),\star}}
 + \enorm{\phi_{j,0}^{(j-1),\star} - \phi_{j-1}^{(j-1)}}
 \\
 &\le \enorm{\phi_\star^{(j)} - \phi_\star^{(j-1)}}
  + \enorm{\phi_\star^{(j-1)} - \phi_{j-1}^{(j-1)}}.
\end{align*}
Combining these two estimates, we see that
\begin{align*}
 \enorm{\phi_\star^{(j)} - \phi_{j,0}^{(j)}}
 &\le (2+\Cinit)\,\enorm{\phi_\star^{(j)} - \phi_\star^{(j-1)}}
 + (1+\Cinit)\, \enorm{\phi_\star^{(j-1)} - \phi_{j-1}^{(j-1)}}
 \\&
 \reff{eq:step1:termination}\le
 (2+\Cinit)\,\enorm{\phi_\star^{(j)} - \phi_\star^{(j-1)}} + \Ci\,\gamma^{j-1}.
\end{align*}
With stability of $\slo^{-1}$ and $\dlo$, Proposition~\ref{prop:uzawa:ideal} (where $\kappa=\gamma>q$ is used) proves that
\begin{align*}
 &\enorm{\phi_\star^{(j)}-\phi_\star^{(j-1)}}
 \simeq \norm{(\dlo-1/2)\gamma_0(u^{(j-1)}_{j-1}-u^{(j-2)}_{j-2})}{H^{1/2}(\Gamma)}
 \lesssim \norm{u^{(j-1)}_{j-1}-u^{(j-2)}_{j-2}}{H^1(\Omega)}
 \\&\qquad
 \le \norm{u-u^{(j-1)}_{j-1}}{H^1(\Omega)} + \norm{u-u^{(j-2)}_{j-2}}{H^1(\Omega)}
 \reff{eq:uzawa:ideal}\lesssim (1+\gamma^{-1})\,\gamma^{j-1}.
\end{align*}
The hidden constants depend only on $\alpha$, $\Ci$, $\Cii$, $\gamma$, and $\Gamma$.
Overall, we thus obtain that
$$\enorm{\phi_\star^{(j)}-\phi_{j,0}^{(j)}}\lesssim \gamma^{j-1},
$$%
where the hidden constant depends only on $\alpha$, $\Ci$, $\Cii$, $\Cinit$, $\gamma$, and $\Gamma$.

{\bf Step~2a.}
Since Algorithm~\ref{algorithm:adaptive} terminated in the $(j-1)$-th iteration, the stopping criterion~\eqref{eq:algorithm:stopping} with $\tau=\gamma^{j-1}$ implies that $\mu_{j-1}(\phi_{j-1}^{(j-1)}) \le \gamma^{j-1}$.
To simplify notation, let $h_\bullet\in L^\infty(\Omega)$ be the local mesh-size, $h_\bullet|_E = |E|^{1/(d-1)}$. Recall that $\TT_{j,0}=\TT_{j-1} = \TT_{j-1}^{\rm[ii]}\in\refine(\TT_{j-1}^{\rm[i]})$ and that $\mu_{j-1}(\cdot)$ is associated with $\TT_{j-1}^{\rm[i]}$. Since $\TT_{j-1} = \TT_{j-1}^{\rm[ii]}\in\refine(\TT_{j-1}^{\rm[i]})$, this proves that 
\begin{align*}
&\Big(\norm{h_{j-1}^{1/2}\nabla_\Gamma\big((\dlo-1/2)(\gamma_0u^{(j-2)}_{j-2}-u_{0,j-2})-V\phi^{(j-1)}_{j-1}\big)}{L^2(\Gamma)}^2 
+ \norm{h_{j-1}^{1/2}(1-\Pi_{j-1})\nabla_\Gamma u_0}{L^2(\Gamma)}^2\Big)^{1/2}
\\&\quad
\le \mu_{j-1}(\phi_{j-1}^{(j-1)}) \le \gamma^{j-1}.
\end{align*}

{\bf Step~2b.}
We employ the local inverse estimate for $\dlo$ from~\cite{invest} to see that
\begin{align*}
  &\norm{h_{j-1}^{1/2}\nabla_\Gamma(\dlo-1/2)\big(\gamma_0(u^{(j-1)}_{j-1}-u^{(j-2)}_{j-2})-(u_{0,j-1}-u_{0,j-2})\big)}{L^2(\Gamma)}
 \\&\quad
 \lesssim \norm{\gamma_0(u^{(j-1)}_{j-1}-u^{(j-2)}_{j-2})-(u_{0,j-1}-u_{0,j-2})}{H^{1/2}(\Gamma)}
 \\&\quad 
 \le \norm{u^{(j-1)}_{j-1}-u^{(j-2)}_{j-2}}{H^1(\Omega)} + \norm{u_{0,j-1}-u_{0,j-2}}{H^{1/2}(\Gamma)}.
\end{align*}
As in Step~1, the first term is estimated by 
$\norm{u^{(j-1)}_{j-1}-u^{(j-2)}_{j-2}}{H^1(\Omega)}\lesssim(1+\gamma^{-1})\gamma^{j-1}.$ 
The second term is estimated with~\eqref{eq:lemma:bem:estimator} as
$\norm{u_{0,j-1}-u_{0,j-2}}{H^{1/2}(\Gamma)} \lesssim \mu_{j-2}(\phi_{j-2}^{(j-2)}) \le \gamma^{j-2}.$
Altogether, we obtain
$$
\norm{h_{j-1}^{1/2}\nabla_\Gamma(\dlo-1/2)\big(\gamma_0(u^{(j-1)}_{j-1}-u^{(j-2)}_{j-2})-(u_{0,j-1}-u_{0,j-2})\big)}{L^2(\Gamma)}
\lesssim (1+\gamma^{-1})\gamma^{j-1},
$$
where the hidden constant depends only on $\alpha$, $\Ci$, $\Cii$, $\Cinit$, $\gamma$, $p$, $\Gamma$, and 
$\sigma$-shape regularity of $\TT_j$.

{\bf Step~2c.} 
We employ the local inverse estimate for $\slo$ from~\cite{invest} to see that
\begin{align*}
 &\norm{h_{j-1}^{1/2}\nabla_\Gamma\slo(\phi_{j,0}^{(j)}-\phi_{j-1}^{(j-1)}}{L^2(\Gamma)}
 \le \norm{\phi_{j,0}^{(j)}-\phi_{j-1}^{(j-1)}}{H^{-1/2}(\Gamma)}
  = \enorm{\phi_{j,0}^{(j)}-\phi_{j-1}^{(j-1)}}.
\end{align*}
As in Step~1, it holds that
\begin{align*}
 \enorm{\phi_{j,0}^{(j)}-\phi_{j-1}^{(j-1)}}
 \le \enorm{\phi_{j,0}^{(j)}-\phi_\star^{(j)}}
 + \enorm{\phi_\star^{(j)} - \phi_\star^{(j-1)}}
 + \enorm{\phi_\star^{(j-1)} - \phi_{j-1}^{(j-1)}}
 \lesssim \gamma^{j-1},
\end{align*}
where the hidden constant depends only on $\alpha$, $\Ci$, $\Cii$, $\Cinit$, $\gamma$, $\sigma$-shape regularity of $\TT_j$, the polynomial degree $p$, and on $\Gamma$.

{\bf Step~2d.}
Recall that $h_{j,0} = h_{j-1}$. The combination of Step~2a--2c proves
{\small\begin{align*}
 &\mu_{j,0}(\phi_{j,0}^{(j)})
 = \Big(\norm{h_{j-1}^{1/2}\nabla_\Gamma\big((\dlo\!-\!1/2)(\gamma_0u^{(j-1)}_{j-1}\!-\!u_{0,j-1})\!-\!V\phi^{(j-1)}_{j,0}\big)}{L^2(\Gamma)}^2 
+ \norm{h_{j-1}^{1/2}(1\!-\!\Pi_{j-1})\nabla_\Gamma u_0}{L^2(\Gamma)}^2\Big)^{1/2}
 \\&\quad\le
 \Big(\norm{h_{j-1}^{1/2}\nabla_\Gamma\big((\dlo-1/2)(\gamma_0u^{(j-2)}_{j-2}-u_{0,j-2})-V\phi^{(j-1)}_{j-1}\big)}{L^2(\Gamma)}^2 
+ \norm{h_{j-1}^{1/2}(1-\Pi_{j-1})\nabla_\Gamma u_0}{L^2(\Gamma)}^2\Big)^{1/2}
 \\&\qquad\qquad
 + \norm{h_{j-1}^{1/2}\nabla_\Gamma(\dlo-1/2)\big(\gamma_0(u^{(j-1)}_{j-1}-u^{(j-2)}_{j-2})-(u_{0,j-1}-u_{0,j-2})\big)}{L^2(\Gamma)}
 \\&\qquad\qquad
  + \norm{h_{j-1}^{1/2}\nabla_\Gamma\slo(\phi_{j,0}^{(j)}-\phi_{j-1}^{(j-1)})}{L^2(\Gamma)}
 \\&\quad\lesssim \gamma^{j-1}. 
\end{align*}}%
Overall, the combination of Step~1 and Step~2d verifies~\eqref{eq0:prop:bem:estimator} and hence concludes the proof.

\end{proof}

\subsection{Realization of step~[ii] of Uzawa iteration}
\label{section:uzawa_step_ii}%
Step~[ii] of Algorithm~\ref{algorithm:uzawa:ideal} will be realized by means of Algorithm~\ref{algorithm:adaptive}, where 
$$
\psi := w_\star^{(j)} \in H^1(\Omega)=:\HH,
\quad 
\XX_\bullet := \SS^{p}(\TT_\bullet), 
\quad\text{and}\quad 
\edual\cdot\cdot:
= \ip{\Riesz(\cdot)}{(\cdot)}_\Omega.
$$
Note that $\enorm\cdot = \norm\cdot{H^1(\Omega)}$.
We employ a weighted-residual error estimator similar to, e.g.,~\cite{aoAposteriori,verfuerth}. 
We suppose additional regularity $\phi_0\in L^2(\Gamma)$.
Recall that
$\Riesz w_\star^{(j)} = f + \gamma_0'\phi_0 - (\FE u_{j-1}^{(j-1)} -\gamma_0' \phi_{j}^{(j)}) \in \widetilde
H^{-1}(\Omega)$. 
Therefore, the estimator reads
\begin{align*}
  \eta_\bullet(T,\psi_\bullet)^2 &:= |T|^{2/d} \norm{f+\div(\material(\nabla u_{j-1}^{(j-1)}) + \nabla \psi_\bullet) 
  -b(\nabla u_{j-1}^{(j-1)}) - c(u_{j-1}^{(j-1)}) - \psi_\bullet }{L^2(T)}^2 \\
  &\qquad + |T|^{1/d} \norm{ [(\material(\nabla u_{j-1}^{(j-1)})+\nabla\psi_\bullet)\cdot\normal]}{L^2(\partial T \setminus\Gamma)}^2 \\
  &\qquad + |T|^{1/d} \norm{\phi_0+\phi_j^{(j)} - (\material(\nabla u_{j-1}^{(j-1)})+\nabla\psi_\bullet)\cdot\normal}{L^2(\partial T\cap \Gamma)}^2.
\end{align*}
The following observation goes back to~\cite{ckns}, where the properties~\eqref{axiom:stability}--\eqref{axiom:reduction} are implicitly proved in~\cite[Section~3.1]{ckns}.

\begin{lemma}\label{lemma:fem:estimator}
The error estimator $\eta_\bullet(\cdot)$ satisfies the assumptions~\eqref{axiom:stability}--\eqref{axiom:reliability} from Section~\ref{section:adaptive}, where $\Cstab$, $\Crel$, and $0<\qred<1$ depend only on the mesh-refinement strategy $\refine(\cdot)$ and $\sigma$-shape regularity of $\TT_0$.
\end{lemma}

\begin{proof}[Sketch of proof]
Throughout, we suppose that $\TT_\circ\in\refine(\TT_\bullet)$ and $\TT_\bullet\in\refine(\TT_0)$.
To see~\eqref{axiom:stability}, let $T\in\TT_\circ\cap\TT_\bullet$ and $\psi_\circ\in\SS^p(\TT_\circ)$,
$\psi_\bullet\in\SS^p(\TT_\bullet)$. Then,
\begin{align*}
  |\eta_\circ(T,\psi_\circ)-\eta_\bullet(T,\psi_\bullet)|^2 &\leq |T|^{2/d}
  \norm{\Delta(\psi_\circ-\psi_\bullet)-(\psi_\circ-\psi_\bullet)}{L^2(T)}^2 
  \\ &\quad + |T|^{1/d}
  \big(\norm{[\nabla(\psi_\circ-\psi_\bullet)\cdot\normal]}{L^2(\partial T\backslash\Gamma)}^2
  +\norm{\nabla (\psi_\circ-\psi_\bullet)\cdot\normal}{L^2(\partial T\cap\Gamma)}^2\big).
 \end{align*}
Together with an inverse inequality and the trace inequality, we hence obtain that
\begin{align*}
  &|\eta_\circ(\TT_\circ\cap\TT_\bullet,\psi_\circ)-\eta_\bullet(\TT_\circ\cap\TT_\bullet,\psi_\bullet)|^2
  \\
  &\qquad \leq \sum_{T\in\TT_\circ\cap\TT_\bullet} \Big( |T|^{2/d} \norm{\Delta(\psi_\circ-\psi_\bullet)-(\psi_\circ-\psi_\bullet)}{L^2(T)}^2 + 2\,|T|^{1/d}
  \norm{\nabla(\psi_\circ-\psi_\bullet)}{L^2(\partial T)}^2 \Big)
  \\
  &\qquad \leq \Cstab \enorm{\psi_\circ-\psi_\bullet}^2.
\end{align*}
The constant $\Cstab>$ depends only on
the mesh-refinement strategy $\refine(\cdot)$, $\sigma$-shape regularity of $\TT_0$, and the polynomial degree $p$.

To see~\eqref{axiom:reduction}, note that all local contributions to the estimator $\eta_\bullet(\psi_\bullet)^2$
are weighted with either $|T|^{2/d}$ or $|T|^{1/d}$.
Therefore,~\eqref{axiom:reduction} simply follows from reduction of refined elements; see Section~\ref{sec:refine}.

Finally, reliability~\eqref{axiom:reliability} follows with the same techniques as in~\cite{aoAposteriori,verfuerth}.
The only difference is that we have to tackle the term $\FE u_{j-1}^{(j-1)}$ from the right-hand side. In general, this
term is not in $L^2(\Omega)$. 
However, elementwise integration by parts proves
\begin{align*}
  &\ip{\FE u_{j-1}^{(j-1)}}v_\Omega = \sum_{T\in\TT_\bullet} \Big(\ip{\material(\nabla u_{j-1}^{(j-1)})}{\nabla v}_T 
  + \ip{b(\nabla u_{j-1}^{(j-1)}) + c(u_{j-1}^{(j-1)})}{v}_T\Big) 
  \\&\quad
  = \sum_{T\in\TT_\bullet} \Big(\ip{-\div\material(\nabla u_{j-1}^{(j-1)})}{v}_T + \ip{\material(\nabla u_{j-1}^{(j-1)})\cdot\normal}{v}_{\partial T} 
 + \ip{b(u_{j-1}^{(j-1)}) + c(u_{j-1}^{(j-1)})}{v}_T\Big)
\end{align*}
for all $v\in H^1(\Omega)$. With this identity, the residual $\Riesz w_\star^{j}-f -\gamma_0'\phi_{0}
+\FE(u_{j-1}^{(j-1)})-\gamma_0'\phi_{j}^{(j)}$ can be estimated with standard techniques.
\end{proof}

Because of Lemma~\ref{lemma:fem:estimator}, we can employ Algorithm~\ref{algorithm:adaptive} to realize step~[ii] of
Algorithm~\ref{algorithm:uzawa:ideal}. 
Moreover, provided that the inverse-type inequality
\begin{align}\label{eq:invest:fem}
  &\sum_{T\in\TT_\bullet} \Big(|T|^{2/d} \norm{\div(\material(\nabla v_\bullet)-\material(\nabla w_\bullet)) - b(\nabla v_\bullet) + b(\nabla w_\bullet) - c(v_\bullet)+c(w_\bullet)}{L^2(T)}^2  \nonumber \\
  &\qquad  + |T|^{1/d} \norm{ [\material(\nabla v_\bullet)-\material(\nabla w_\bullet)]\cdot\normal}{L^2(\partial T\setminus \Gamma)}^2 \nonumber
  + |T|^{1/d} \norm{ (\material(\nabla v_\bullet)-\material(\nabla w_\bullet))\cdot\normal}{L^2(\partial T\cap \Gamma)}^2 \Big)
  \\&\qquad \lesssim \norm{v_\bullet-w_\bullet}{H^1(\Omega)}^2
\end{align}
holds for all $v_\bullet,w_\bullet \in \SS^p(\TT_\bullet)$ with some hidden constant that depends only on $\material$, $b$, $c$, the polynomial degree $p$, and $\sigma$-shape regularity of $\TT_\bullet$, the following theorem proves that the number of adaptive
iterations of Algorithm~\ref{algorithm:adaptive} is uniformly bounded.

\begin{theorem}\label{prop:fem:estimator}
Let $0<\gamma<1$, $0<\theta\le1$, and $0\leq \lcg<1$. 
For $j\in\N$, choose $\tau = \gamma^j$ and
$\TT'_0:=\TT_j^{\rm{[i]}}$. Then, the following assertions {\rm(a)--(b)} hold:

{\rm(a)}
After $\ell\in\N_0$ iterations, Algorithm~\ref{algorithm:adaptive} returns the triangulation $\TT_j^{\rm[ii]} = \TT'_\ell
\in \refine(\TT_j^{\rm[i]})$ and a corresponding discrete function 
$w_j^{(j)}\in\SS^{p}(\TT_\ell')$ such that 
\begin{align}\label{eq:step1:termination:fem} 
  \norm{w^{(j)}_\star-w_j^{(j)}}{H^1(\Omega)}\le\Cii\,\gamma^j,
\end{align}
where $\Cii>0$ depends only on $\Crel$, $\Cstab$, and $\Gamma$. 

{\rm(b)} Suppose that $\alpha>0$ is sufficiently small in the sense of Proposition~\ref{prop:uzawa:continuous} and that
$0<q<1$ ist the resulting contraction constant of the continuous Uzawa iteration. Suppose $q<\gamma$. Moreover, suppose that there exists $\Cinit>0$ such that for all $j\ge2$, the initial guess $w_{j,0}^{(j)}\in\SS^p(\TT_j^\mathrm{[i]})$ for Algorithm~\ref{algorithm:adaptive} satisfies
\begin{align}
  \norm{w_{j,0}^{(j),\star}-w_{j,0}^{(j)}}{H^1(\Omega)}
  \le \Cinit\, \norm{w_{j,0}^{(j),\star}-w_{j-1}^{(j-1)}}{H^1(\Omega)},
\end{align}
where $w_{j,0}^{(j),\star}\in\SS^p(\TT_j^\mathrm{[i]})$ is the best approximation of $w_\star^{(j)}$ in $\SS^p(\TT_j^\mathrm{[i]})$ with respect to $\norm\cdot{H^1(\Omega)}$. Finally, suppose
that~\eqref{eq:invest:fem} holds.
Then, the number $\ell\in\N_0$ of iterations in Algorithm~\ref{algorithm:adaptive} is uniformly bounded for all
$j\in\N$, i.e., $\ell\le L$, where $L>0$ depends only on $\alpha$, $\Ci$, $\Cii$, $\Cinit$, $\gamma$, $p$ as well as on uniform
$\sigma$-shape regularity of the triangulation $\TT_j^\mathrm{[i]}\in\refine(\TT_0)$ and on $\Gamma$.
\end{theorem}

\begin{remark}
  We note that the additional assumption~\eqref{eq:invest:fem} on $\FE$ is satisfied if $\FE$ is linear with
  coefficients $\material \in W^{1,\infty}(\Omega;\R^{d\times d})$, $b\in L^\infty(\Omega;\R^d)$, and $c\in L^\infty(\Omega)$.
\end{remark}

\begin{proof}[Proof of Theorem~\ref{prop:fem:estimator}]
To prove~(a), note that $\enorm{v} := \ip{\Riesz v}{v}_\Omega^{1/2} = \norm{v}{H^1(\Omega)}$ for all $v\in H^1(\Omega)$. 
The first claim together with estimate~\eqref{eq:step1:termination:fem} follows from Proposition~\ref{theorem:adaptive}, where $\Cii \simeq \Cstp$.

To prove~(b), we argue as in the proof of Proposition~\ref{prop:bem:estimator}. We may
suppose $j\ge2$. Let $\TT_{j,\ell} = \TT_\ell'$ be the $\ell$-th adaptive mesh in Algorithm~\ref{algorithm:adaptive} (in the $j$-th iteration of Algorithm~\ref{algorithm:uzawa:ideal}). 
Let $w^{(j)}_{j,\ell} \in \SS^p(\TT_{j,\ell})$ be the corresponding approximation.
Recall that Algorithm~\ref{algorithm:adaptive} guarantees that
\begin{align*}
 \eta_{j,\ell}(w^{(j)}_{j,\ell})^2
 \le \kctr^{-1}\qctr^\ell\,\big(\norm{w^{j}_\star-w^{(j)}_{j,0}}{H^1(\Omega)} + \kctr\eta_{j,0}(w^{(j)}_{j,0})\big)^2
 \quad\text{for all }\ell\ge0.
\end{align*}
To conclude the proof of~(b), it only remains to show that %
\begin{align}\label{eq0:prop:fem:estimator}
 \norm{w^{j}_\star-w^{(j)}_{j,0}}{H^1(\Omega)} + \eta_{j,0}(w^{(j)}_{j,0})
 \le C\,\gamma^{j-1}
\end{align}
where $C>0$ is independent of $j$. 

{\bf Step~1.}
Lipschitz continuity of $\FE$, the definitions of the dual norms $\norm\cdot{\widetilde H^{-1}(\Omega)}$, $\norm\cdot{H^{-1/2}(\Gamma)}$, and Proposition~\ref{prop:uzawa:ideal} (where $\kappa=\gamma>q$) prove that
\begin{align*}
  \norm{w_\star^{(j)}-w_\star^{(j-1)}}{H^1(\Omega)} &= \norm{\Riesz(w_\star^{(j)}-w_\star^{(j-1)})}{\widetilde H^{-1}(\Omega)} \\
  &= \norm{\FE u_{j-2}^{(j-2)}-\FE u_{j-1}^{(j-1)} -\gamma_0'\phi_{j-1}^{(j-1)} + \gamma_0'\phi_j^{(j)}}{\widetilde H^{-1}(\Omega)} \\
  &\lesssim \norm{u_{j-2}^{(j-2)}- u_{j-1}^{(j-1)}}{H^1(\Omega)} + \norm{\phi_{j-1}^{(j-1)}-\phi_j^{(j)}}{H^{-1/2}(\Gamma)} 
  \lesssim\gamma^{j-1}.
\end{align*}
The hidden constants depend only on $\alpha$, $\Ci$, $\Cii$, $\gamma$, and $\FE$.
Arguing as in step~1 of the proof of Proposition~\ref{prop:bem:estimator}, we conclude $\norm{w_\star^{(j)}-w_{j,0}^{(j)}}{H^1(\Omega)} \lesssim \gamma^{j-1}$.

{\bf Step~2.}
As above, the stopping criterion~\eqref{eq:algorithm:stopping} implies that $\eta_{j-1}(w_{j-1}^{(j-1)}) \le \gamma^{j-1}$.
Following the proof of Proposition~\ref{prop:bem:estimator}, we obtain that
\begin{align*}
  \eta_{j,0}(w_{j,0}^j) = \eta_{j,0}(w_{j-1}^{(j-1)}) \leq \eta_{j-1}(w_{j-1}^{(j-1)}) + R
\end{align*}
with 
\begin{align*}
  R &:= \sum_{T\in\TT_j^\mathrm{[i]}} \Big( 
  |T|^{2/d} \norm{\Delta(w_{j-1}^{(j-1)}-w_{j,0}^{(j)}) + (w_{j-1}^{(j-1)}-w_{j,0}^{(j)})}{L^2(T)}^2
  \\&\qquad
 + |T|^{2/d} \norm{\div(\material(u_{j-1}^{(j-1)})-\material(u_{j-2}^{(j-2)})) - b(\nabla u_{j-1}^{(j-1)}) + b(\nabla u_{j-2}^{(j-2)}) -c(u_{j-1}^{(j-1)}) + c(u_{j-2}^{(j-2)})}{L^2(T)}^2 \\
  &\qquad + |T|^{1/d} \norm{ [\material(u_{j-1}^{(j-1)})-\material(u_{j-2}^{(j-2)})]\cdot\normal}{L^2(\partial T\setminus \Gamma)}^2 \\
  &\qquad + |T|^{1/d} \norm{ (\material(u_{j-1}^{(j-1)})-\material(u_{j-2}^{(j-2)}))\cdot\normal}{L^2(\partial T\cap \Gamma)}^2 \\
  &\qquad + |T|^{1/d} \norm{\phi_{j}^j-\phi_{j-1}^{(j-1)}}{L^2(\partial T\cap\Gamma)}^2 
  \Big).
\end{align*}
Note that $u_{j-1}^{(j-1)} - u_{j-2}^{(j-2)} \in \SS^p(\TT_j^\mathrm{[i]})$ and $\phi_{j}^{(j)}-\phi_{j-1}^{(j-1)} \in \PP^{p-1}(\TT_j^\mathrm{[i]}|_\Gamma)$. Together with an inverse inequality and the assumption~\eqref{eq:invest:fem}, this proves that
\begin{align*}
  R \lesssim \norm{u_{j-1}^{(j-1)}-u_{j-2}^{(j-2)}}{H^1(\Omega)}^2 + \norm{\phi_j^{(j)}-\phi_{j-1}^{(j-1)}}{H^{-1/2}(\Gamma)}^2 
  + \eta_{j-1}(w_{j-1}^{(j-1)}) + \mu_{j-1}(\phi_{j-1}^{(j-1)}).
\end{align*}
Hence,
\begin{align*}
  \eta_{j,0}(w_{j-1}^{(j-1)}) &\lesssim \eta_{j-1}(w_{j-1}^{(j-1)}) + \mu_{j-1}(\phi_{j-1}^{(j-1)})  + 
  \norm{u_{j-1}^{(j-1)}-u_{j-2}^{(j-2)}}{H^1(\Omega)}^2 + \norm{\phi_j^{(j)}-\phi_{j-1}^{(j-1)}}{H^{-1/2}(\Gamma)}^2 \\
  &\lesssim \gamma^{j-1}.
\end{align*}
This finishes the proof.
\end{proof}

\subsection{Global a posteriori error estimate}
\label{section:globalEstimate}%
\def\JN{\mathcal{B}}
\def\Cglobal{C_\mathrm{glo}}
In this section, we derive a global upper bound for the error $\norm{u-u_{j}^{(j)}}{H^1(\Omega)} +
\norm{\phi-\phi_{j}^{(j)}}{H^{-1/2}(\Gamma)}$.
To that end, let %
\begin{align*}
\JN : H^1(\Omega)\times H^{-1/2}(\Gamma) \to \widetilde H^{-1}(\Omega)\times H^{1/2}(\Gamma), \quad
  \JN(u,\phi) = (\FE u -\gamma_0'\phi,(\tfrac12-\dlo)\gamma_0 u + \slo\phi)
\end{align*}
denote the operator associated to the Johnson-N\'ed\'elec coupling~\eqref{eq:jn}.
Let $w_j^{(j),\star} \in \SS^p(\TT_j^\mathrm{[ii]})$ be the best approximation of $w_\star^{(j)}$ with respect to
$\norm{\cdot}{H^1(\Omega)}$ and let $\phi_j^{(j),\star}\in \PP^{p-1}(\TT_j^\mathrm{[i]}|_\Gamma)$ be the best
approximation of $\phi_\star^{(j)}$ with respect to the $\ip{\cdot}{\slo(\cdot)}_\Gamma$ induced norm.
The upper bound in the next theorem involves the terms
\begin{align*}
  \norm{w_j^{(j),\star}-w_j^{(j)}}{H^1(\Omega)} \quad\text{and}\quad
  \norm{\phi_j^{(j),\star}-\phi_j^{(j)}}{H^{-1/2}(\Gamma)},
\end{align*}
which stem from the fact that we use inexact solvers. In our setting, these terms are evaluated within the PCG
algorithm and hence known a posteriori terms; see Remark~\ref{remark:reliable}.
\begin{theorem}\label{thm:globalEstimate}
  Suppose that $\FE$ is 
  \begin{itemize}
    \item either linear with $c_\FE>0$
    \item or strongly monotone with $c_\FE>c_\dlo/4$,
  \end{itemize}
  where $c_\dlo \in [1/2,1)$ denotes the contraction constant of the double-layer integral operator.
  Then, there holds
  \begin{align*}
    &\Cglobal^{-1}\big(\norm{u-u_{j}^{(j)}}{H^1(\Omega)} + \norm{\phi-\phi_{j}^{(j)}}{H^{-1/2}(\Gamma)}\big) 
    \\&\leq \nu_j := \eta_j(w_j^{(j)}) + \mu_j(\phi_j^{(j)}) + \norm{w_j^{(j)}}{H^1(\Omega)}
    +\norm{w_j^{(j),\star}-w_j^{(j)}}{H^1(\Omega)}
    +\norm{\phi_j^{(j),\star}-\phi_j^{(j)}}{H^{-1/2}(\Gamma)}
  \end{align*}
  The constant $\Cglobal>0$ depends only on $c_\FE$, $C_\FE$, $\Omega$, $\Crel$, $\Cstab$, $\sigma$-shape regularity of
  $\TT_0$, and $\alpha>0$.
\end{theorem}
\begin{proof}
  Define $\HH := H^1(\Omega)\times H^{-1/2}(\Gamma)$ and let $\dual\cdot\cdot$ denote the duality pairing
  between $\HH$ and its dual $\HH^*$.
  We stress that $\JN^{-1}$ exists and is Lipschitz continuous, i.e.,
  \begin{align}\label{eq:jnlip}
    \norm{\uu-\vv}{\HH} \lesssim 
    \norm{\JN\uu-\JN\vv}{\HH^*} \quad\text{for all }\uu = (u,\phi), \vv = (v,\psi) \in \HH.
  \end{align}
  To see~\eqref{eq:jnlip} in the case that $\FE$ is strongly monotone with $c_\FE>c_\dlo/4$, one follows~\cite[Section~5.1]{convBEFE}.
  If $\FE$ is linear with $c_\FE>0$, the Johnson-N\'ed\'elec coupling~\eqref{eq:jn} is equivalent to the model
  problem~\eqref{eq:modell}. Note that $(\tfrac12-\dlo)$ is bijective, so that the right-hand side in~\eqref{eq:jn} is
  an arbitrary functional on $\HH$. Therefore, unique solvability of the model
  problem~\eqref{eq:modell} proves that $\JN$ is bijective. The inverse mapping theorem thus implies~\eqref{eq:jnlip}.

  Recall the Richardson iteration  $u_j^{(j)} = u_{j-1}^{(j-1)} + \alpha w_{j}^{(j)}$ in the Uzawa algorithm.
  Therefore,
  \begin{align*}
    \norm{u-u_j^{(j)}}{H^1(\Omega)} \leq \norm{u-u_{j-1}^{(j-1)}}{H^1(\Omega)} + 
    \alpha \, \norm{w_j^{(j)}}{H^1(\Omega)}.
  \end{align*}
  Lipschitz continuity of $\JN^{-1}$ shows that
  \begin{align*}
    &\norm{u-u_{j-1}^{(j-1)}}{H^1(\Omega)} + \norm{\phi-\phi_{j}^{(j)}}{H^{-1/2}(\Gamma)}  
    \simeq \norm{(u,\phi)-(u_{j-1}^{(j-1)},\phi_j^{(j)})}{\HH} \\
    &\quad \lesssim \norm{\JN(u,\phi) -\JN(u_{j-1}^{(j-1)},\phi_j^{(j)})}{\HH^*}   \\
    &\quad \simeq \norm{\FE u \!-\! \FE u_{j-1}^{(j-1)} \!-\! \gamma_0'(\phi-\phi_j^{(j)})}{\widetilde H^{-1}(\Omega)} 
     + \norm{\slo(\phi-\phi_j^{(j)}) \!+\! (\tfrac12-\dlo)\gamma_0(u-u_{j-1}^{(j-1)})}{H^{1/2}(\Gamma)}.
  \end{align*}
  To estimate the boundary contribution, recall that $\slo\phi_\star^{(j)} = (\dlo-\tfrac12)(\gamma_0 u_{j-1}^{(j-1)}-u_0)$ and $\slo\phi =
  (\dlo-\tfrac12)(\gamma_0 u- u_0)$. Together with Remark~\ref{remark:reliable}, this proves that
  \begin{align*}
    &\norm{\slo(\phi-\phi_j^{(j)}) + (\tfrac12-\dlo)\gamma_0(u-u_{j-1}^{(j-1)})}{H^{1/2}(\Gamma)} = 
    \norm{\slo(\phi_\star^{(j)}-\phi_j^{(j)})}{H^{1/2}(\Gamma)} \\&\qquad \simeq
    \norm{\phi_\star^{(j)}-\phi_j^{(j)}}{H^{-1/2}(\Gamma)} \lesssim \mu_j(\phi_j^{(j)}) 
    + \norm{\phi_j^{(j),\star}-\phi_j^{(j)}}{H^{-1/2}(\Gamma)}.
  \end{align*}
  To estimate the volume contribution, recall that $\Riesz w_\star^{(j)} = f + \gamma_0' \phi_0 - (\FE u_{j-1}^{(j-1)} -
  \gamma_0'\phi_j^{(j)})$ and $\FE u -\gamma_0'\phi = f+\gamma_0'\phi_0$. 
  This gives 
  \begin{align*}
    \norm{\FE u - \FE u_{j-1}^{(j-1)} - \gamma_0'(\phi-\phi_j^{(j)})}{\widetilde H^{-1}(\Omega)} 
    &= \norm{\Riesz w_\star^{(j)}}{\widetilde H^{-1}(\Omega)} = \norm{w_\star^{(j)}}{H^1(\Omega)}.
  \end{align*}
  With the triangle inequality and Remark~\ref{remark:reliable}, we further infer that
  \begin{align*}
    \norm{w_\star^{(j)}}{H^1(\Omega)} &\leq \norm{w_\star^{(j)}-w_j^{(j)}}{H^1(\Omega)} + \norm{w_j^{(j)}}{H^1(\Omega)}
    \\
    &\lesssim  \eta_j(w_j^{(j)}) + \norm{w_j^{(j),\star}-w_j^{(j)}}{H^1(\Omega)} + \norm{w_j^{(j)}}{H^1(\Omega)}.
  \end{align*}
  Putting all together, we conclude the proof.
\end{proof}

\section{Numerical examples}\label{sec:ex}
\begin{figure}[htb]
  \begin{center}
    \begin{tikzpicture}
\begin{axis}[
    axis equal,
    width=0.45\textwidth,
    xlabel={$x$},
    ylabel={$y$},
]

\addplot[patch,color=white,
faceted color = black, line width = 1.5pt,
patch table ={./figures/elements.dat}] file{./figures/coordinates.dat};
\addplot[mark=*,color=gray,only marks] table[x index=0,y index=1] {./figures/coordinates.dat};
\end{axis}
\end{tikzpicture}
    \begin{tikzpicture}
\begin{axis}[
    axis equal,
    width=0.45\textwidth,
    xlabel={$x$},
    ylabel={$y$},
]

\addplot[patch,color=white,
faceted color = black, line width = 1.5pt,
patch table ={./figures/elementsZshape.dat}] file{./figures/coordinatesZshape.dat};
\addplot[mark=*,color=gray,only marks] table[x index=0,y index=1] {./figures/coordinatesZshape.dat};
\end{axis}
\end{tikzpicture}
  \end{center}
  \caption{L-shaped domain with initial triangulation of $12$ elements (left) and Z-shaped domain with initial
  triangulation of $14$ elements (right).}
  \label{fig:Lshape}
\end{figure}
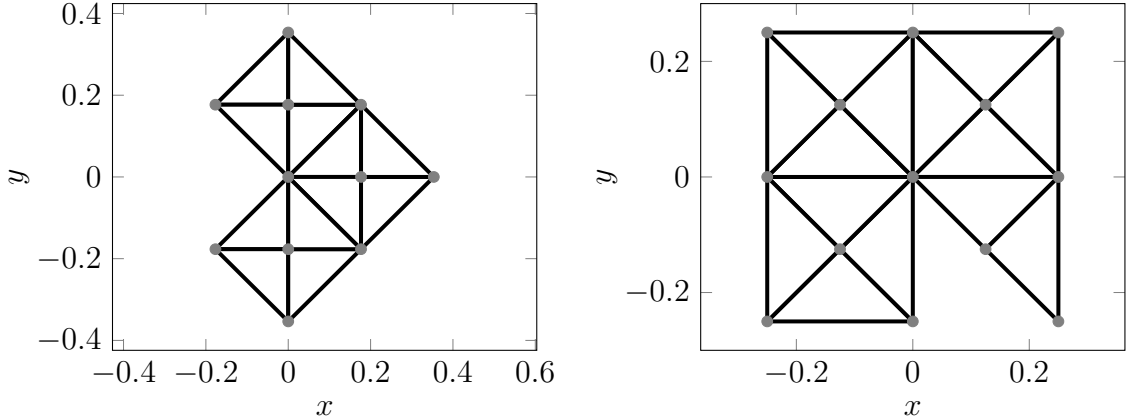
In this section, we present three numerical experiments in 2D to underpin our theoretical findings.
We consider two linear problems on an L-shaped domain and a nonlinear problem on a Z-shaped
domain, sketched in Figure~\ref{fig:Lshape}.
In all examples, the exact solution $(u, u^\mathrm{ext})$ of~\eqref{eq:modell} is known and the solution $u$ in the
interior has a singularity at the reentrant corner. 
We compute the error quantities (in each step of Algorithm~\ref{algorithm:uzawa:ideal})
\begin{align*}
  \err_\Omega^{(j)} &:= \norm{u-u_j^{(j)}}{H^1(\Omega)}, \\
  \err_\Gamma^{(j)} &:= \norm{h_j^{1/2}(\phi-\phi_{j}^{(j)})}{L^2(\Gamma)},
\end{align*}
and compare them to the error estimators $\eta_j,\mu_j$ and the global error estimator $\nu_j$.
Here, $h_j\in L^\infty(\Gamma)$ denotes the local mesh-size $h_j|_E = |E|$.
In all convergence plots, we use triangles to visualize slopes $(\#\TT)^{-s}$, where the experimental convergence rate
$s>0$ is written besides the
triangle.
The parameter for the D\"orfler marking criterion~\eqref{eq:abstract:doerfler} is set to $\theta =0.25$.

\subsection{Laplace transmission problem}\label{sec:ex1}
\begin{figure}[htb]
  \begin{center}
    \begin{tikzpicture}
\begin{loglogaxis}[
    title={\tiny $\gamma=0.85$},
width=0.49\textwidth,
cycle list name=exotic,
xlabel={number of elements $\#\TT$},
grid=major,
legend entries={\tiny $\err_\Omega$,\tiny $\err_\Gamma$, \tiny $\eta$, \tiny $\mu$, \tiny $\nu$},
legend pos=south west,
]
\addplot table [x=nE,y=errUZAWAH1] {figures/example1_gamma_85.dat};
\addplot table [x=nE,y=errUZAWABEM] {figures/example1_gamma_85.dat};
\addplot table [x=nE,y=estFEM] {figures/example1_gamma_85.dat};
\addplot table [x=nE,y=estBEM] {figures/example1_gamma_85.dat};
\addplot table [x=nE,y=estTOT] {figures/example1_gamma_85.dat};

\logLogSlopeTriangle{0.9}{0.2}{0.35}{0.5}{black}{{\tiny $\tfrac12$}};
\logLogSlopeTriangle{0.9}{0.3}{0.62}{0.2}{black}{{\tiny $\tfrac15$}};
\logLogSlopeTriangleBelow{0.85}{0.2}{0.05}{0.75}{black}{{\tiny $\tfrac34$}};
\end{loglogaxis}
\end{tikzpicture}
    \begin{tikzpicture}
\begin{loglogaxis}[
    title={\tiny $\gamma=0.9$},
width=0.49\textwidth,
cycle list name=exotic,
xlabel={number of elements $\#\TT$},
grid=major,
legend entries={\tiny $\err_\Omega$,\tiny $\err_\Gamma$, \tiny $\eta$, \tiny $\mu$, \tiny $\nu$},
legend pos=south west,
]
\addplot table [x=nE,y=errUZAWAH1] {figures/example1_gamma_90.dat};
\addplot table [x=nE,y=errUZAWABEM] {figures/example1_gamma_90.dat};
\addplot table [x=nE,y=estFEM] {figures/example1_gamma_90.dat};
\addplot table [x=nE,y=estBEM] {figures/example1_gamma_90.dat};
\addplot table [x=nE,y=estTOT] {figures/example1_gamma_90.dat};

\logLogSlopeTriangle{0.9}{0.15}{0.325}{0.5}{black}{{\tiny $\tfrac12$}};
\logLogSlopeTriangle{0.9}{0.2}{0.5}{0.33}{black}{{\tiny $\tfrac13$}};
\logLogSlopeTriangleBelow{0.85}{0.2}{0.05}{0.75}{black}{{\tiny $\tfrac34$}};
\end{loglogaxis}
\end{tikzpicture}
    \begin{tikzpicture}
\begin{loglogaxis}[
    title={\tiny $\gamma=0.95$},
width=0.49\textwidth,
cycle list name=exotic,
xlabel={number of elements $\#\TT$},
grid=major,
legend entries={\tiny $\err_\Omega$,\tiny $\err_\Gamma$, \tiny $\eta$, \tiny $\mu$, \tiny $\nu$},
legend pos=south west,
]
\addplot table [x=nE,y=errUZAWAH1] {figures/example1_gamma_95.dat};
\addplot table [x=nE,y=errUZAWABEM] {figures/example1_gamma_95.dat};
\addplot table [x=nE,y=estFEM] {figures/example1_gamma_95.dat};
\addplot table [x=nE,y=estBEM] {figures/example1_gamma_95.dat};
\addplot table [x=nE,y=estTOT] {figures/example1_gamma_95.dat};

\logLogSlopeTriangle{0.9}{0.2}{0.4}{0.5}{black}{{\tiny $\tfrac12$}};
\logLogSlopeTriangleBelow{0.85}{0.2}{0.05}{0.75}{black}{{\tiny $\tfrac34$}};
\end{loglogaxis}
\end{tikzpicture}
    \begin{tikzpicture}
\begin{loglogaxis}[
    title={\tiny $\gamma=0.98$},
width=0.49\textwidth,
cycle list name=exotic,
xlabel={number of elements $\#\TT$},
grid=major,
legend entries={\tiny $\err_\Omega$,\tiny $\err_\Gamma$, \tiny $\eta$, \tiny $\mu$, \tiny $\nu$},
legend pos=south west,
]
\addplot table [x=nE,y=errUZAWAH1] {figures/example1_gamma_98.dat};
\addplot table [x=nE,y=errUZAWABEM] {figures/example1_gamma_98.dat};
\addplot table [x=nE,y=estFEM] {figures/example1_gamma_98.dat};
\addplot table [x=nE,y=estBEM] {figures/example1_gamma_98.dat};
\addplot table [x=nE,y=estTOT] {figures/example1_gamma_98.dat};

\logLogSlopeTriangle{0.9}{0.2}{0.4}{0.5}{black}{{\tiny $\tfrac12$}};
\logLogSlopeTriangleBelow{0.85}{0.2}{0.05}{0.75}{black}{{\tiny $\tfrac34$}};
\end{loglogaxis}
\end{tikzpicture}
  \end{center}
  \caption{Error quantities and error estimators over the number of elements for the example from Section~\ref{sec:ex1}.}
  \label{fig:ex1}
\end{figure}
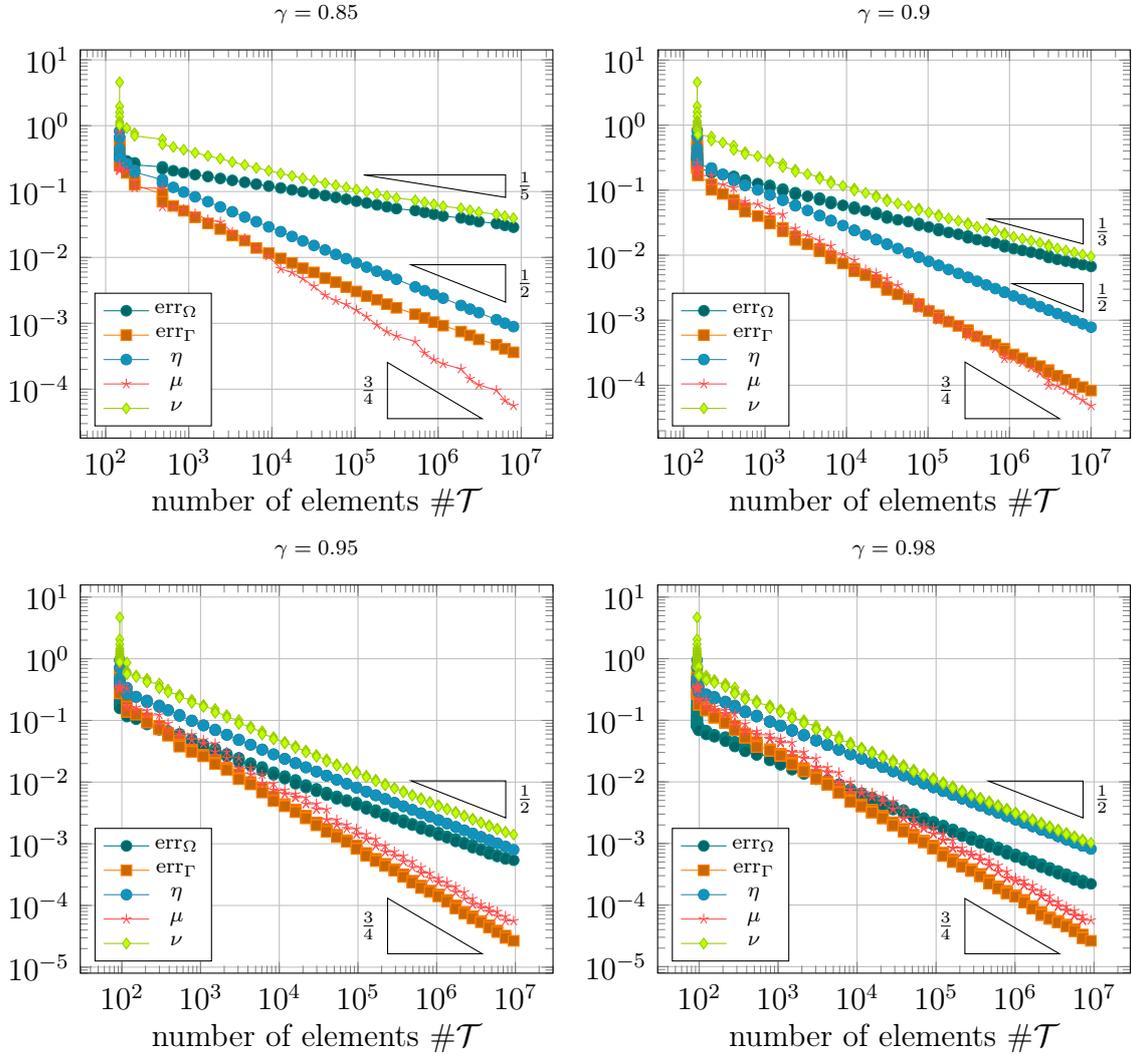
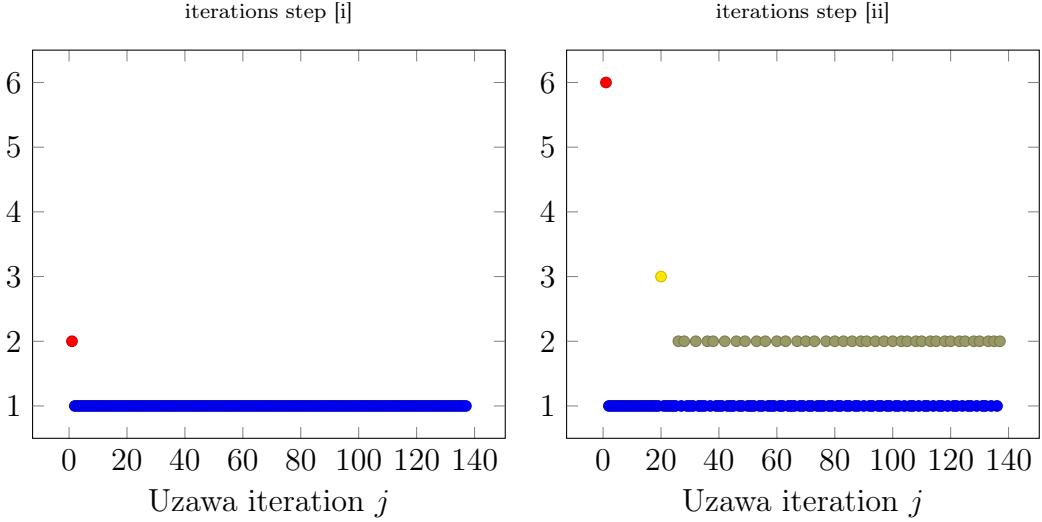
\begin{figure}[htb]
  \begin{center}
    \begin{tikzpicture}
\begin{axis}[
  title={\tiny iterations step [i] },
width=0.49\textwidth,
xlabel={Uzawa iteration $j$},
ymin=0.5,
ymax=6.5,
ytick=\empty,
extra y ticks={1,...,6},
]
  \addplot[scatter, only marks] table [x=iterUZ,y=kBEM] {figures/example1_gamma_95.dat};

\end{axis}
\end{tikzpicture}
    \begin{tikzpicture}
\begin{axis}[
  title={\tiny iterations step [ii] },
width=0.49\textwidth,
xlabel={Uzawa iteration $j$},
ytick=\empty,
extra y ticks={1,...,6},
]
  \addplot[scatter, only marks] table [x=iterUZ,y=kFEM] {figures/example1_gamma_95.dat};

\end{axis}
\end{tikzpicture}
  \end{center}
  \caption{Number of BEM (left) and FEM (right) iterations in each step of the Uzawa-type
  Algorithm~\ref{algorithm:uzawa:ideal} for the example from Section~\ref{sec:ex1} with $\alpha=0.05$, $\gamma=0.95$,
  $\theta=0.25$, and $\tau=10^{-3}$.}
  \label{fig:ex1:iter}
\end{figure}
We choose $u = r^{2/3}\cos(2/3\varphi)$, where $(r,\varphi)$ denote the polar coordinates. The exterior solution
$u^\mathrm{ext}$ is a smooth function.
We set $\material = \II$, $b=0$, $c=0$. Hence, the operator $\FE$ simplifies to
\begin{align*}
  \ip{\FE u}{v}_\Omega = \ip{\nabla u}{\nabla v}_\Omega \quad\text{for all }u,v\in H^1(\Omega),
\end{align*}
which corresponds to the case of the Laplace transmission problem. The approximation order of the spaces is set to
$p=0$, i.e., we use the spaces $\SS^1(\TT)$ and $\PP^0(\TT|_\Gamma)$.
Figure~\ref{fig:ex1} shows the error quantities, and estimators over the number of elements for a fixed $\alpha = 0.05$
and $\gamma \in \{0.85,0.9,0.95,0.98\}$.
We observe suboptimal rates $s=\tfrac15$, resp. $\tfrac13$ for $\gamma = 0.85$ resp. $\gamma = 0.9$, whereas 
$\gamma=0.95,0.98$ lead to the optimal rate $s=\tfrac12$ for the overall error.
This indicates that the contraction constant from Proposition~\ref{prop:uzawa:ideal} satisfies $q>\gamma=0.85,0.9$ and $q<0.95,0.98$. Comparing the number of total iterations
($j$) in the Uzawa-type Algorithm~\ref{algorithm:uzawa:ideal}, we get $j=43,67,137,347$ for $\gamma=0.85,0.9,0.95,0.98$.
Altough in both cases $\gamma=0.95,0.98$ we obtain optimal convergence rates, much more iterations are needed for
$\gamma=0.98$.

For all values of $\gamma$, we observe that the number of iterations within the steps [i] and [ii] of
Algorithm~\ref{algorithm:uzawa:ideal} is bounded, i.e., the number of BEM and FEM problems to be solved is bounded (see
Theorem~\ref{prop:bem:estimator} and Theorem~\ref{prop:fem:estimator}).
Figure~\ref{fig:ex1:iter} visualizes the number of iterations in step [i] and step [ii] with respect to the number of
iterations (indexed with $j$) in Algorithm~\ref{algorithm:uzawa:ideal} for $\gamma=0.95$.
The FEM and BEM problems are solved using PCG with appropriate (local) multilevel additive Schwarz preconditioners and
relative tolerance $\tau=10^{-3}$ (see Section~\ref{section:adaptive}).

\subsection{Modified Uzawa algorithm}\label{sec:uzawa:mod}
The observations of the last section lead us to a modified algorithm (not analyzed), where we adapt the value of $\gamma$
at the end of each iteration.
This modification is based on the following observation, which follows along the proof of
Proposition~\ref{prop:uzawa:ideal}: It holds that
\begin{align*}
  u^{(j+1)} - u^{(j)}  &= (\II-\alpha\Riesz^{-1}\LL)(u^{(j)}-u^{(j-1)}),
\end{align*}
and thus
\begin{align*}
  \norm{u^{(j+1)}-u^{(j)}}{H^1(\Omega)} \leq \norm{\II-\alpha\Riesz^{-1}\LL}{} 
  \norm{u^{(j)}-u^{(j-1)}}{H^1(\Omega)} \leq q \norm{u^{(j)}-u^{(j-1)}}{H^1(\Omega)}.
\end{align*}
This proves that
\begin{align}
  \frac{\norm{u^{(j+1)}-u^{(j)}}{H^1(\Omega)}}{\norm{u^{(j)}-u^{(j-1)}}{H^1(\Omega)}} \leq q.
\end{align}
Since the exact iterates $u^{(j)}$ are not known, we replace them by our approximations $u_{j}^{(j)}$ in the following
algorithm. Recall that $w_j^{(j)} = (u_j^{(j)}-u_{j-1}^{(j-1)})/\alpha$.

\begin{algorithm}\label{algorithm:uzawa:mod}
{\bfseries Input:} Parameter $\alpha > 0$, initial triangulation $\TT_0$, initial guess $u^{(0)}_0\in \XX_0$, constants
  $\Ci,\Cii>0$, $\eps_1>0$ and $0<\gamma<1$.\\
{\bfseries Discrete Uzawa iteration:} For all $j=1,2,3,\dots$, iterate the following steps~{\rm[i]--[v]}:
\begin{itemize}
\item[{\rm[i]}] Determine $\TT^{\rm[i]}_{j} \in \refine(\TT_{j-1})$ as well as some $\phi_j^{(j)}\in\PP^{p-1}(\TT^{\rm[i]}_{j}|_\Gamma)$ such that
\begin{align}\label{eq1:uzawa:mod}
 \norm{\phi_\star^{(j)} - \phi_j^{(j)}}{H^{-1/2}(\Gamma)} 
 \le \Ci \, \eps_j.
\end{align}
\item[{\rm[ii]}] Determine a triangulation $\TT^{\rm[ii]}_{j} \in \refine(\TT_j^{\rm[i]})$ and some $w_j^{(j)}\in\SS^{p}(\TT^{\rm[ii]}_{j})$ such that 
\begin{align}\label{eq2:uzawa:mod}
 \norm{w^{(j)}_\star - w^{(j)}_j}{H^1(\Omega)}
 \le\Cii\,\eps_j
\end{align}
\item[{\rm[iii]}] Define $\TT_j := \TT^{\rm[ii]}_{j}$ and $u_j^{(j)}:=u_{j-1}^{(j-1)}+\alpha w_j^{(j)}\in\SS^p(\TT_j)$.
\item[{\rm[iv]}] If $j\geq2$ update $\gamma := \norm{w_j^{(j)}}{H^1(\Omega)}/\norm{w_{j-1}^{(j-1)}}{H^1(\Omega)}$.
\item[{\rm[v]}] Set $\eps_{j+1} = \gamma\,\eps_{j}$.
\end{itemize}
\end{algorithm}

\subsection{Laplace transmission problem with modified Algorithm}\label{sec:ex2}
We consider the same example and configuration as in Section~\ref{sec:ex1}, but now use the modified
Algorithm~\ref{algorithm:uzawa:mod} instead of Algorithm~\ref{algorithm:uzawa:ideal}.
We set $\eps_1=1$ with initial $\gamma=0.95$. Figure~\ref{fig:ex2} shows the error quantities and error estimators over
the number of elements for $\alpha=0.03$ and $\alpha=0.05$. 
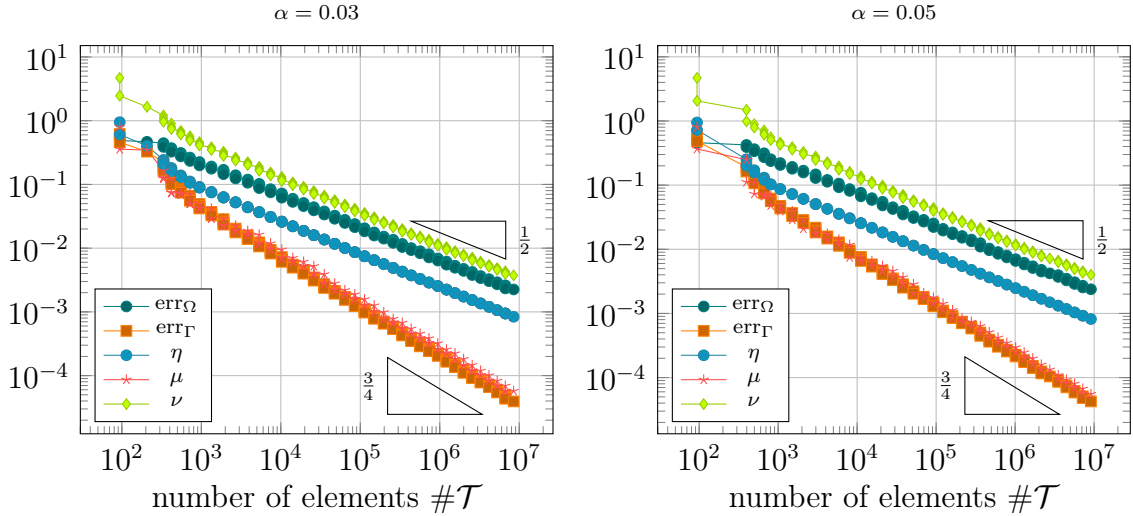
\begin{figure}[htb]
  \begin{center}
    \begin{tikzpicture}
\begin{loglogaxis}[
    title={\tiny $\alpha=0.03$},
width=0.49\textwidth,
cycle list name=exotic,
xlabel={number of elements $\#\TT$},
grid=major,
legend entries={\tiny $\err_\Omega$,\tiny $\err_\Gamma$, \tiny $\eta$, \tiny $\mu$, \tiny $\nu$},
legend pos=south west,
]
\addplot table [x=nE,y=errUZAWAH1] {figures/example2_alpha_003.dat};
\addplot table [x=nE,y=errUZAWABEM] {figures/example2_alpha_003.dat};
\addplot table [x=nE,y=estFEM] {figures/example2_alpha_003.dat};
\addplot table [x=nE,y=estBEM] {figures/example2_alpha_003.dat};
\addplot table [x=nE,y=estTOT] {figures/example2_alpha_003.dat};

\logLogSlopeTriangle{0.9}{0.2}{0.45}{0.5}{black}{{\tiny $\tfrac12$}};
\logLogSlopeTriangleBelow{0.85}{0.2}{0.05}{0.75}{black}{{\tiny $\tfrac34$}};
\end{loglogaxis}
\end{tikzpicture}
    \begin{tikzpicture}
\begin{loglogaxis}[
    title={\tiny $\alpha=0.05$},
width=0.49\textwidth,
cycle list name=exotic,
xlabel={number of elements $\#\TT$},
grid=major,
legend entries={\tiny $\err_\Omega$,\tiny $\err_\Gamma$, \tiny $\eta$, \tiny $\mu$, \tiny $\nu$},
legend pos=south west,
]
\addplot table [x=nE,y=errUZAWAH1] {figures/example2_alpha_005.dat};
\addplot table [x=nE,y=errUZAWABEM] {figures/example2_alpha_005.dat};
\addplot table [x=nE,y=estFEM] {figures/example2_alpha_005.dat};
\addplot table [x=nE,y=estBEM] {figures/example2_alpha_005.dat};
\addplot table [x=nE,y=estTOT] {figures/example2_alpha_005.dat};

\logLogSlopeTriangle{0.9}{0.2}{0.45}{0.5}{black}{{\tiny $\tfrac12$}};
\logLogSlopeTriangleBelow{0.85}{0.2}{0.05}{0.75}{black}{{\tiny $\tfrac34$}};
\end{loglogaxis}
\end{tikzpicture}
  \end{center}
  \caption{Error quantities and error estimators over the number of elements for the example from Section~\ref{sec:ex2}.}
  \label{fig:ex2}
\end{figure}
We observe that both values of $\alpha$ lead to optimal rates $s=\tfrac12$ for the overall error with respect to
the number of volume elements. For $\alpha=0.05$, we obtain a total of $82$ iterations, which is much less in comparison
to the $137$ iterations obtained in Section~\ref{sec:ex1} with a fixed $\gamma=0.95$.

\subsection{Transmission problem using higher order elements}\label{sec:ex3}
We consider essentially the same setting as in Section~\ref{sec:ex2}, but now use higher order elements $p=2$.
We set $\alpha = 0.07$ and use an exact solver for the realization of step [i] and [ii] in Algorithm~\ref{algorithm:uzawa:mod}. Furthermore, we
replace $\FE$ by the operator
\begin{align*}
  \ip{\FE u}v_\Omega := \tfrac1{10} \ip{\nabla u}{\nabla v}_\Omega \quad\text{for all } u,v\in H^1(\Omega).
\end{align*}
Note that for this operator there holds $c_\FE = \tfrac1{10}< c_\dlo/4$. We stress that it is not known whether the
discrete Johnson-N\'ed\'elec coupling~\eqref{eq:jn:discrete} is solvable or not; see Remark~\ref{rem:jn:discrete}.
Figure~\ref{fig:ex34} (left plot) shows the error quantities and error estimators with respect to the number of elements. 
As before, we observe an optimal decay of the overall error with respect to the number of elements, which in the case
$p=2$ is $s=1$.
\begin{figure}[htb]
  \begin{center}
    \begin{tikzpicture}
\begin{loglogaxis}[
    title={\tiny P2-P1 elements},
width=0.49\textwidth,
cycle list name=exotic,
xlabel={number of elements $\#\TT$},
grid=major,
legend entries={\tiny $\err_\Omega$,\tiny $\err_\Gamma$, \tiny $\eta$, \tiny $\mu$, \tiny $\nu$},
legend pos=south west,
]
\addplot table [x=nE,y=errUZAWAH1] {figures/example3.dat};
\addplot table [x=nE,y=errUZAWABEM] {figures/example3.dat};
\addplot table [x=nE,y=estFEM] {figures/example3.dat};
\addplot table [x=nE,y=estBEM] {figures/example3.dat};
\addplot table [x=nE,y=estTOT] {figures/example3.dat};

\logLogSlopeTriangle{0.9}{0.2}{0.33}{1}{black}{{\tiny $1$}};
\logLogSlopeTriangleBelow{0.85}{0.2}{0.05}{1.25}{black}{{\tiny $\tfrac54$}};
\end{loglogaxis}
\end{tikzpicture}
    \begin{tikzpicture}
\begin{loglogaxis}[
    title={\tiny Nonlinear},
width=0.49\textwidth,
cycle list name=exotic,
xlabel={number of elements $\#\TT$},
grid=major,
legend entries={\tiny $\err_\Omega$,\tiny $\err_\Gamma$, \tiny $\eta$, \tiny $\mu$, \tiny $\nu$},
legend pos=south west,
]
\addplot table [x=nE,y=errUZAWAH1] {figures/example4.dat};
\addplot table [x=nE,y=errUZAWABEM] {figures/example4.dat};
\addplot table [x=nE,y=estFEM] {figures/example4.dat};
\addplot table [x=nE,y=estBEM] {figures/example4.dat};
\addplot table [x=nE,y=estTOT] {figures/example4.dat};

\logLogSlopeTriangle{0.9}{0.2}{0.4}{0.5}{black}{{\tiny $\tfrac12$}};
\logLogSlopeTriangleBelow{0.85}{0.2}{0.05}{0.75}{black}{{\tiny $\tfrac34$}};
\end{loglogaxis}
\end{tikzpicture}
  \end{center}
  \caption{Error quantities and error estimators over the number of elements for the example from Section~\ref{sec:ex3}
  (left) resp. Section~\ref{sec:ex4} (right).}
  \label{fig:ex34}
\end{figure}
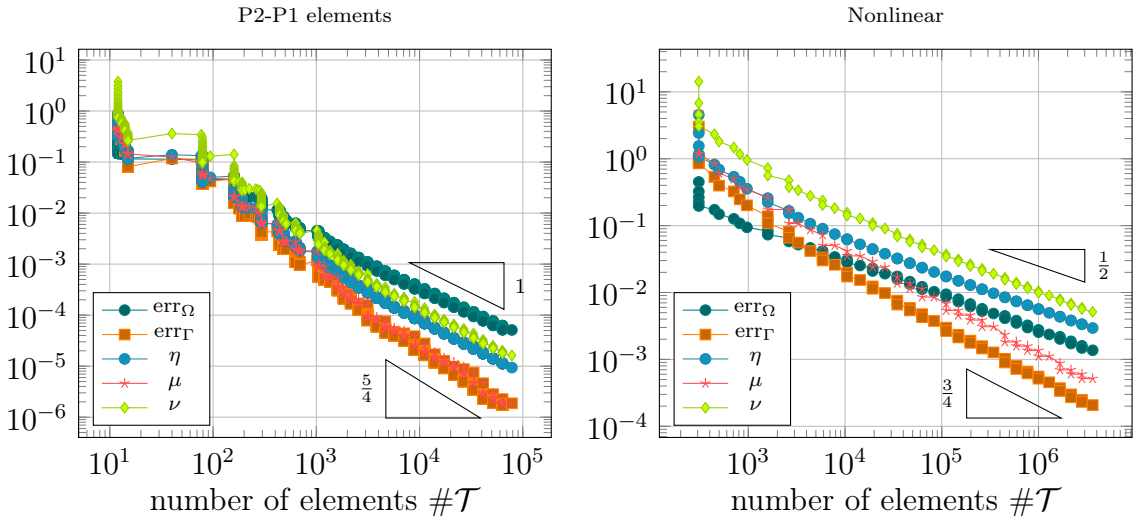

\subsection{Nonlinear transmission problem on Z-shaped domain}\label{sec:ex4}
We consider a nonlinear transmission problem on the Z-shaped domain sketched in Figure~\ref{fig:Lshape}. Again we apply
Algorithm~\ref{algorithm:uzawa:mod} with $\alpha=0.07$, $\eps_1 = 5$ and initial $\gamma=0.95$. We use lowest-order elements $p=1$ for
the approximations and an exact solver to realize step [i] and [ii] of Algorithm~\ref{algorithm:uzawa:mod}.
We set $c=b=0$ and $A\nabla v = \chi(|\nabla v|)\nabla v$ with $\chi(t) = 1+ \tanh(t)/t$ for $t>0$
and $\chi(0) = 2$, hence, $\FE$ reads
\begin{align*}
  \ip{\FE u}v_\Omega = \ip{\chi(|\nabla u|)\nabla u}{\nabla v}_\Omega \quad\text{for all }u,v\in H^1(\Omega).
\end{align*}
It can be proved that $\FE$ is strongly monotone with constant $c_\FE = 1$. We prescribe the exact interior solution $u =
r^{4/7}\cos(4/7\varphi)$ in polar coordinates $(r,\varphi)$ and use a smooth function for the exterior
solution $u^\mathrm{ext}$.
The error quantities and error estimators over the number of elements are plotted in Figure~\ref{fig:ex34} (right plot).
Again we observe an optimal decay of the overall error.


\bibliographystyle{abbrv}
\bibliography{literature}

\end{document}